\documentclass{article}
\usepackage{amsfonts}
\usepackage{amsmath}
\usepackage{amssymb}
\usepackage{graphicx}
\usepackage{caption}
\usepackage[numbers]{natbib}

\RequirePackage[OT1]{fontenc}
\newtheorem{theorem}{Theorem}

\newtheorem{claim}[theorem]{Claim}

\newtheorem{lemma}[theorem]{Lemma}

\newtheorem{remark}[theorem]{Remark}

\newenvironment{proof}[1][Proof]{\noindent\textbf{#1.} }{\ \rule{0.5em}{0.5em}}

\title{Long runs under point conditioning. The real case.\protect}
\author{Broniatowski Michel and Caron Virgile}
\begin{document}
\maketitle

\begin{abstract}
This paper presents a sharp approximation of the density of long runs of a random walk conditioned on its end value or by an average of a functions of its summands as  their number tends to infinity. The conditioning event is of moderate or large deviation type. The result extends the Gibbs conditional principle in the sense that it provides a description of the distribution of the random walk on long subsequences. An algorithm for the simulation of such long runs is presented, together with an algorithm determining their maximal length for which the approximation is valid up to a prescribed accuracy.
\end{abstract}

\section{Introduction and notation}

\subsection{Context and scope}

This paper explores the asymptotic distribution of a random walk conditioned
on its final value as the number of summands increases. Denote $\mathbf{X}%
_{1}^{n}:=\left( \mathbf{X}_{1}\mathbf{,..,X}_{n}\right) $ \ a set of $n$
independent copies of a real random variable $\mathbf{X}$ with density $p$
on $\mathbb{R}$ and $\mathbf{S}_{1}^{n}:=\mathbf{X}_{1}+.$\textbf{..}$+%
\mathbf{X}_{n}.$ \ We consider approximations of the density of the vector $%
\mathbf{X}_{1}^{k}\mathbf{=}\left( \mathbf{X}_{1}\mathbf{,..,X}_{k}\right) $
on $\mathbb{R}^{k}$ when $\mathbf{S}_{1}^{n}=n\left( a_{n}\sqrt{Var\mathbf{X}%
}+E\mathbf{X}\right)$ and $a_{n}$ is either fixed different from $0$ or
tends slowly to $0$ and $k:=k_{n}$ is an integer sequence such that 
\begin{equation}
0\leq\lim\sup_{n\rightarrow\infty}k/n\leq1  \label{k/nTENDSTO1}
\end{equation}
together with 
\begin{equation}
\lim_{n\rightarrow\infty}n-k=\infty.  \label{pas trop vite}
\end{equation}
Therefore we may consider the asymptopic behavior of the density of the
trajectory of the random walk on long runs. For sake of applications we also
address the case when $\mathbf{S}_{1}^{n}$ is substituted by $\mathbf{U}%
_{1}^{n}:=f\left( \mathbf{X}_{1}\right) +...+f\left( \mathbf{X}_{n}\right) $
for some real valued measurable function $f$, and when the conditioning
event writes $\left( \mathbf{U}_{1}^{n}=n\left( a_{n}\sqrt{Varf\left( 
\mathbf{X}\right) }+Ef\left( \mathbf{X}\right) \right) \right) .$

The interest in this question stems from various sources. When $k$ is fixed
(typically $k=1$) this is a version of the \textit{Gibbs Conditional
Principle } which has been studied extensively for fixed $a_{n}$, therefore
under a \textit{large deviation} condition. Diaconis and Freedman \cite%
{DiaconisFreedman} have considered this issue also in the case $%
k/n\rightarrow\theta$ for $0\leq\theta<1$, in connection with de Finetti's
Theorem for exchangeable finite sequences. Their interest was related to the
approximation of the density of $\mathbf{X}_{1}^{k}$ by the \textit{product
density }of the summands $\mathbf{X}_{i}$'s, therefore on the permanence of
the independence of the $\mathbf{X}_{i}$'s under conditioning. Their result
is in the spirit of van Camperhout and Cover \cite{VanCamperhoutCover1981}
and to be paralleled with Csiszar's \cite{Csiszar1084} asymptotic
conditional independence result, when the conditioning event is $\left( 
\mathbf{S}_{1}^{n}>n\left( a_{n}\sqrt{Var\mathbf{X}}+E\mathbf{X}%
\right)\right)$ with $a_{n}$ fixed and positive. In the same vein and under
the same \textit{large deviation} condition Dembo and Zeitouni \cite%
{DemboZeitouni1996} considered similar problems. This question is also of
importance in Statistical Physics. Numerous papers pertaining to structural
properties of polymers deal with this issue, and we refer to \cite%
{denHollanderWeiss1988} and \cite{denHollanderWeiss1988a} for a description
of those problems and related results. In the moderate deviation case
Ermakov \cite{ERmakov2006} also considered a similar problem when $k=1.$
Although out of the scope of the present paper the result which is presented
here is a cornerstone in the development of fast Importance Sampling
procedures for rare event simulation; see a first attempt in this direction
in \cite{BroniatowskiRitov2009}. In Statistics, $M-$ estimators have the
same weak behavior as the empirical mean of their influence functions on the
sampling points in the moderate deviation zone. Simulating samples under a
given value of the $M-$estimator leads to improved test procedures under
small $p-$values.

We exhibit the change in the dependence structure of the $\mathbf{X}_{i}$'s
under the conditioning as $k/n\rightarrow1$ and provide an explicit and
constructive solution to the approximation scheme. The approximating density
is obtained as an adaptive change in the classical tilting argument combined
with an adaptive change in the variance. Also when $k=o(n)$ our result
improves on existing ones since it provides a sharp approximation of the
conditional density. The present result is optimal in the sense that it
coincides with the exact conditional density in the gaussian case.

The crucial aspect of our result is the following.\ The approximation of the
density of $\mathbf{X}_{1}^{k}$ is not performed on the sequence of entire
spaces $\mathbb{R}^{k}$ but merely on a sequence of subsets of $\mathbb{R}%
^{k}$ which bear the trajectories of the conditioned random walk with
probability going to $1$ as $n$ tends to infinity; therefore the
approximation is performed on \textit{typical paths}. The reason which led
us to consider approximation in this peculiar sense is twofold. First the
approximation on typical paths is what is in fact needed for the
applications of the present results in the field of simulation and of rare
event analysis; second it avoids a number of technical conditions which are
necessary in order to get an approximation on all $\mathbb{R}^{k}$ and which
are indeed central in the above mentioned works; those conditions pertain to
the regularity of the characteristic function of the underlying density $p$
in order to get a good approximation in remote regions of $\mathbb{R}^{k}$.
Since the approximation is handled on paths generated under the conditional
density of the $\mathbf{X}_{i}$'s under the conditioning, much is known on
the region of $\mathbb{R}^{k}$ which is reached with large probability by
the conditioned random walk, through the analysis of the large values of the 
$\mathbf{X}_{i}$'s.

For sake of numerical applications we provide explicit algorithms for the
generation of such random walks together with a number of comments for the
practical implementation. Also an explicit rule for the maximal value of $k$
compatible with a given accuracy of the approximating scheme is presented
and numerical simulation supports this rule; an algorithm for its
calculation is presented.

\subsection{Notation and hypotheses}

In the context of the point conditioning 
\begin{equation*}
\mathcal{E}_{n}:=\left( \mathbf{S}_{1}^{n}=n\left( a_{n}\sqrt{Var\mathbf{X}}%
+E\mathbf{X}\right) \right)
\end{equation*}%
the hypotheses are as below. The case when $\mathbf{S}_{1}^{n}$ is
substituted by $\mathbf{U}_{1}^{n}$ is postponed to Section 3, together with
the relevant hypotheses and notation.

We assume that $\mathbf{X}$ satisfies the Cramer condition, i.e. $\mathbf{X} 
$ has a finite moment generating function $\Phi (t):=E\exp t\mathbf{X}$ in a
non void neighborhood of $0;$ denote 
\begin{equation*}
m(t):=\frac{d}{dt}\log \Phi (t)
\end{equation*}%
and 
\begin{equation*}
s^{2}(t):=\frac{d}{dt}m(t).
\end{equation*}%
The values of $m(t)$ and $s^{2}$ are the expectation and the variance of the 
\textit{tilted} density 
\begin{equation}
\pi ^{\alpha }(x):=\frac{\exp tx}{\Phi (t)}p(x)  \label{tilted density}
\end{equation}%
where $t$ is the only solution of the equation $m(t)=\alpha $ when $\alpha $
belongs to the support of $\mathbf{X}$, see Barnfoff-Nielsen \cite{Barndorff1978} for details. Denote $\Pi ^{\alpha }$ the probability measure with density $\pi ^{\alpha }$.

We also assume that the characteristic function of $\mathbf{X}$ is in $L^{r}$
for some $r\geq 1$ which is necessary for the Edgeworth expansions to be
performed.

The probability measure of the random vector $\mathbf{X}_{1}^{n}$ on $%
\mathbb{R}^{n}$ conditioned upon $\mathcal{E}_{n}$ is denoted $\mathfrak{P}%
_{n}.$ We also denote $\mathfrak{P}_{n}$ the corresponding distribution of $%
\mathbf{X}_{1}^{k}$ conditioned upon $\mathcal{E}_{n}$; the vector $\mathbf{X%
}_{1}^{k}$ then has a density with respect to the Lebesgue measure on $%
\mathbb{R}^{k}$ for $1\leq k<n$ ,which will be denoted $\mathfrak{p}_{n}$,
which might seem ambiguous but recalls that the conditioned distribution
pertains to the value of $\mathbf{S}_{1}^{n},$ from which the density of $%
\mathbf{X}_{1}^{k}$ is obtained. For a generic r.v. $\mathbf{Z}$ with
density $p$, we denote $p\left( \mathbf{Z}=z\right)$ the value of $p$ at
point $z.$

\bigskip

This paper is organized as follows. Section 2 presents the approximation
scheme for the conditional density of $\mathbf{X}_{1}^{k}$ under the point
conditioning sequence $\mathcal{E}_{n}.$ In section 3, it is extended to the
case when the conditioning family of events writes $\left( \mathbf{U}%
_{1}^{n}=n\left( a_{n}\sqrt{Varf\left( \mathbf{X}\right) }+Ef\left( \mathbf{X%
}\right) \right) \right) .$ The value of $k$ for which this approximation is
fair is discussed; an algorithm for the implementation of this rule is
proposed. Section 4 presents an algorithm for the simulation of random
variables under the approximating scheme. We have kept the main steps of the
proofs in the core of the paper; some of the technicalities is left to the
Appendix.

\bigskip

\section{Random walks conditioned on their sum}

We introduce a positive sequence $\epsilon _{n}$\ which satisfies

\begin{align}
\lim_{n\rightarrow\infty}\epsilon_{n}\sqrt{n-k} & =\infty  \tag{E1} \\
\lim_{n\rightarrow\infty}\epsilon_{n}\left( \log n\right) ^{2} & =0.\text{ }
\tag{E2}
\end{align}

It will be shown that $\epsilon_{n}\left( \log n\right) ^{2}$ is the rate of
accuracy of the approximating scheme.

We denote $a$ the generic term of the bounded sequence $\left( a_{n}\right)
_{n\geq1},$ which we assume positive, without loss of generality $\mathbf{.}$
The event $\mathcal{E}_{n}$ is of moderate or large deviation type, since we
assume that 
\begin{equation}
\lim_{n\rightarrow\infty}\frac{a^{2}}{\epsilon_{n}\left( \log n\right)^{2} }%
=\infty.  \tag{A}
\end{equation}

The case when $a$ does not depend on $n$ satisfies (A) for any sequence $%
\epsilon_{n}$ under (E1,2).\ Conditions (A) and (E1,2) jointly imply that $a$
cannot satisfy $\sqrt{n}a\rightarrow c$ for some fixed $c;$ the Central
Limit zone is not covered by our result. In order that there exists a
sequence $\epsilon_{n}$ such that the approximation of $\mathfrak{p}_{n}$
holds with rate $\epsilon_{n}\left( \log n\right) ^{2}\rightarrow0$, a
sufficient condition on $a_{n}$ is 
\begin{equation}
\lim_{n\rightarrow\infty}\frac{\sqrt{n}a^{2}}{(\log n)^{2}}=\infty
\label{a_n-e(x)}
\end{equation}
which covers both the moderate and the large deviation cases.

Under these assumptions $k$ can be fixed or can grow together with $n$ with
the restriction that $n-k$ should tend to infinity; when $a$ is fixed this
rate is governed through (E1) (or reciprocally given $k$ ,$\epsilon_{n}$ is
governed by $k$) independently on $a.$ In the moderate deviation case for a
given sequence $a$ close to $0$, $\epsilon_{n}$ has rapid decrease, which in
turn forces $n-k$ to grow rapidly.

In this section we assume that $\mathbf{X}$ has expectation $0$ and variance 
$1.$ For clearness the dependence in $n$ of all quantities involved in the
coming development is omitted in the notation.

\subsection{Approximation of the density of the runs}

Let $a=a_{n}$ denote the current term of a sequence satisfying (A). Define a
density $g_{a}(y_{1}^{k})$ on $\mathbb{R}^{k}$ as follows. Set 
\begin{equation*}
g_{0}(\left. y_{1}\right\vert y_{0}):=\pi^{a}(y_{1})
\end{equation*}
with $y_{0}$ arbitrary, and for $1\leq i\leq k-1$ define $g_{i}(\left.
y_{i+1}\right\vert y_{1}^{i})$ recursively.

Set $t_{i}$ the unique solution of the equation%
\begin{equation}
m_{i}:=m(t_{i})=\frac{n}{n-i}\left( a-\frac{s_{1}^{i}}{n}\right)  \label{mi}
\end{equation}
where $s_{1}^{i}:=y_{1}+...+y_{i}.$ The tilted adaptive family of densities $%
\pi^{m_{i}}$ is the basic ingredient of the derivation of approximating
scheme$.$ Let 
\begin{equation*}
s_{i}^{2}:=\frac{d^{2}}{dt^{2}}\left( \log E_{\pi^{m_{i}}}\exp t\mathbf{X}%
\right) \left( 0\right)
\end{equation*}
and%
\begin{equation*}
\mu_{j}^{i}:=\frac{d^{j}}{dt^{j}}\left( \log E_{\pi^{m_{i}}}\exp t\mathbf{X}%
\right) \left( 0\right) ,\text{ }j=3,4
\end{equation*}
which are the second , third and fourth centered moments of $\pi^{m_{i}}.$
Let%
\begin{equation}
g_{i}(\left. y_{i+1}\right\vert y_{1}^{i})=C_{i}p(y_{i+1})\mathfrak{n}\left(
a+\alpha\beta,\alpha,y_{i+1}\right)  \label{g-i}
\end{equation}
where $\mathfrak{n}\left( \mu,\tau,x\right) $ is the normal density with
mean $\mu$ and variance $\tau$ at $x$. Here 
\begin{equation}
\alpha=s_{i}^{2}\left( n-i-1\right)  \label{a}
\end{equation}%
\begin{equation}
\beta=t_{i}+\frac{\mu_{3}^{i}}{2s_{i}^{2}\left( n-i-1\right) }  \label{b}
\end{equation}
and $C_{i}$ is a normalizing constant.

Define 
\begin{equation}
g_{a}(y_{1}^{k}):=\prod_{i=0}^{k-1}g_{i}(\left. y_{i+1}\right\vert
y_{1}^{i}).  \label{g_s}
\end{equation}

We then have

\begin{theorem}
\label{Prop approx local cond density} Assume that (E1,2) holds together
with (A)$.$ Let $Y_{1}^{n}$ be a sample with distribution $\mathfrak{P}_{n}.$
Then 
\begin{equation}
\mathfrak{p}_{n}\left( Y_{1}^{k}\right) :=p(\left. \mathbf{X}%
_{1}^{k}=Y_{1}^{k}\right\vert \mathbf{S}_{1}^{n}=na)=g_{a}(Y_{1}^{k})(1+o_{%
\mathfrak{P}_{n}}(\epsilon_{n}\left( \log n\right) ^{2})).
\label{local approx under exact value}
\end{equation}
\end{theorem}

\begin{proof}
The proof uses Bayes formula to write $p(\left. \mathbf{X}%
_{1}^{k}=Y_{1}^{k}\right\vert \mathbf{S}_{1}^{n}=na)$ as a product of $k$
conditional densities of individual terms of the trajectory evaluated at $%
Y_{1}^{k}$. Each term of this product is approximated through an Edgeworth
expansion which together with the properties of $Y_{1}^{k}$ under $\mathfrak{%
P}_{n}$ concludes the proof. This proof is rather long and we have differed
its technical steps to the Appendix.

Denote $\Sigma_{1}^{0}=0$, $\Sigma_{1}^{1}:=Y_{1}$\ $\ $\ and $%
\Sigma_{1}^{i}:=\Sigma_{1}^{i-1}+Y_{i}.$ It holds 
\begin{align}
p(\left. \mathbf{X}_{1}^{k}=Y_{1}^{k}\right\vert \mathbf{S}_{1}^{n} &
=na)=p(\left. \mathbf{X}_{1}=Y_{1}\right\vert \mathbf{S}_{1}^{n}=na)
\label{joint density} \\
\prod_{i=1}^{k-1}p(\left. \mathbf{X}_{i+1}=Y_{i+1}\right\vert \mathbf{X}%
_{1}^{i} & =Y_{1}^{i},\mathbf{S}_{1}^{n}=na)  \notag \\
& =\prod_{i=0}^{k-1}p\left( \left. \mathbf{X}_{i+1}=Y_{i+1}\right\vert 
\mathbf{S}_{i+1}^{n}=na-\Sigma_{1}^{i}\right)  \notag
\end{align}
using independence of the r.v's $\mathbf{X}_{i}^{\prime}s$.

We make use of the following property which states the invariance of
conditional densities under the tilting: For $1\leq i\leq j\leq n,$ for all $%
a$ in\ the range of $\mathbf{X},$ for all $u$ and $s$ 
\begin{equation}
p\left( \left. \mathbf{S}_{i}^{j}=u\right\vert \mathbf{S}_{1}^{n}=s\right)
=\pi^{a}\left( \left. \mathbf{S}_{i}^{j}=u\right\vert \mathbf{S}%
_{1}^{n}=s\right) .  \label{inv tilting}
\end{equation}
Define $t_{i}$ through 
\begin{equation*}
m(t_{i})=\frac{n}{n-i}\left( a-\frac{\Sigma_{1}^{i}}{n}\right)
\end{equation*}
a function of the past r.v's $Y_{1}^{i}$ and set $m_{i}:=m(t_{i})$ and $%
s_{i}^{2}:=s^{2}(t_{i}).$ By (\ref{inv tilting}) 
\begin{align*}
& p\left( \left. \mathbf{X}_{i+1}=Y_{i+1}\right\vert \mathbf{S}%
_{i+1}^{n}=na-\Sigma_{1}^{i}\right) \\
& =\pi^{m_{i}}\left( \left. \mathbf{X}_{i+1}=Y_{i+1}\right\vert \mathbf{S}%
_{i+1}^{n}=na-\Sigma_{1}^{i}\right) \\
& =\pi^{m_{i}}\left( \mathbf{X}_{i+1}=Y_{i+1}\right) \frac{\pi^{m_{i}}\left( 
\mathbf{S}_{i+2}^{n}=na-\Sigma_{1}^{i+1}\right) }{\pi^{m_{i}}\left( \mathbf{S%
}_{i+1}^{n}=na-\Sigma_{1}^{i}\right) }
\end{align*}
where we used the independence of the $\mathbf{X}_{j}$'s under $\pi^{m_{i}}.$
A precise evaluation of the dominating terms in this lattest expression is
needed in order to handle the product (\ref{joint density}).

Under the sequence of densities $\pi ^{m_{i}}$ the i.i.d. r.v's $\mathbf{X}%
_{i+1},...,\mathbf{X}_{n}$ define a triangular array which satisfies a local
central limit theorem, and an Edgeworth expansion. Under $\pi ^{m_{i}}$, $%
\mathbf{X}_{i+1}$ has expectation $m_{i}$ and variance $s_{i}^{2}.$ Center
and normalize both the numerator and denominator in the fraction which
appears in the last display. Denote $\overline{\pi _{n-i-1}}$ the density of
the normalized sum $\left( \mathbf{S}_{i+2}^{n}-(n-i-1)m_{i}\right) /\left(
s_{i}\sqrt{n-i-1}\right) $ when the summands are i.i.d. with common density $%
\pi ^{m_{i}}.$ Accordingly $\overline{\pi _{n-i}}$ is the density of $\left( 
\mathbf{S}_{i+1}^{n}-(n-i)m_{i}\right) /\left( s_{i}\sqrt{n-i}\right) $
under i.i.d.\ $\pi ^{m_{i}}$ sampling. Hence, evaluating both $\overline{\pi
_{n-i-1}}$ and its normal approximation at point $Y_{i+1},$ 
\begin{align}
& p\left( \left. \mathbf{X}_{i+1}=Y_{i+1}\right\vert \mathbf{S}%
_{i+1}^{n}=na-\Sigma _{1}^{i}\right)  \label{condTilt} \\
& =\frac{\sqrt{n-i}}{\sqrt{n-i-1}}\pi ^{m_{i}}\left( \mathbf{X}%
_{i+1}=Y_{i+1}\right) \frac{\overline{\pi _{n-i-1}}\left( \left(
m_{i}-Y_{i+1}\right) /s_{i}\sqrt{n-i-1}\right) }{\overline{\pi _{n-i}}(0)} 
\notag \\
& :=\frac{\sqrt{n-i}}{\sqrt{n-i-1}}\pi ^{m_{i}}\left( \mathbf{X}%
_{i+1}=Y_{i+1}\right) \frac{N_{i}}{D_{i}}.  \notag
\end{align}%
The sequence of densities $\overline{\pi _{n-i-1}}$ converges pointwise to
the standard normal density under (E1) which implies that $n-k$ tends to
infinity, and an Edgeworth expansion to the order 5 is performed for the
numerator and the denominator. The main arguments used in order to obtain
the order of magnitude of the envolved quantities are (i) a maximal
inequality which controls the magnitude of $m_{i}$ $\ $for all $i$ between $%
0 $ and $k-1$ (Lemma \ref{LemmaMaxm_in})$,$ (ii) the order of the maximum of
the $Y_{i}^{\prime }s$ (Lemma \ref{Lemma max X_i under conditioning}). As
proved in the Appendix, under (A) 
\begin{equation}
N_{i}=\phi\left( -Y_{i+1}/\left( s_{i}\sqrt{n-i-1}\right) \right)
.A.B+O_{\mathfrak{P}_{n}}\left( \frac{1}{\left( n-i-1\right) ^{3/2}}\right)
\label{num approx fixed i}
\end{equation}
where $\phi$ is the standard normal density and,
\begin{equation}
A:=\left( 1+\frac{aY_{i+1}}{s_{i}^{2}(n-i-1)}-\frac{a^{2}}{2s_{i}^{2}(n-i-1)}%
+\frac{o_{\mathfrak{P}_{n}}(\epsilon _{n}\log n)}{n-i-1}\right)  \label{Adem}
\end{equation}
and 
\begin{equation}
B:=\left( 
\begin{array}{c}
1-\frac{\mu _{3}^{i}}{2s_{i}^{4}\left( n-i-1\right) }(a-Y_{i+1}) \\ 
-\frac{\mu _{3}^{i}-s_{i}^{4}}{8s_{i}^{4}(n-i-1)}-\frac{15(\mu_{3}^{i})^{2}}{%
72s_{i}^{6}(n-i-1)}+\frac{O_{\mathfrak{P}_{n}}\left( (\log n)^{2}\right)}{%
\left(n-i-1\right) ^{2}}%
\end{array}%
\right)  \label{Bdem}
\end{equation}
The $O_{\mathfrak{P}_{n}}\left( \frac{1}{\left( n-i-1\right) ^{3/2}}\right) $
term in (\ref{num approx fixed i}) is uniform upon $\left(
m_{i}-Y_{i+1}\right) /s_{i}\sqrt{n-i-1}.$ Turn back to (\ref{condTilt}) and
do the same Edgeworth expansion in the demominator, which writes%
\begin{equation}
D_{i}=\phi(0)\left( 1-\frac{\mu _{3}^{i}-s_{i}^{4}}{8s_{i}^{4}(n-i)}-%
\frac{15(\mu_{3}^{i})^{2}}{72s_{i}^{6}(n-i)} \right) +O_{\mathfrak{P}_{n}}\left( 
\frac{1}{\left( n-i\right) ^{3/2}}\right) .  \label{PI 0}
\end{equation}%
The terms in $g_{i}(\left. Y_{i+1}\right\vert Y_{1}^{i})$ follow from an
expansion in the ratio of the two expressions (\ref{num approx fixed i}) and
(\ref{PI 0}) above. The gaussian contribution is explicit in (\ref{num
approx fixed i}) while the term $\exp \frac{\mu_{3}^{i}}{2s_{i}^{4}\left(
n-i-1\right) }Y_{i+1}$ is the dominant term in $B$. Turning to (\ref%
{condTilt}) and comparing with (\ref{local approx under exact value}) it
appears that the normalizing factor $C_{i}$ in $g_{i}(\left.
Y_{i+1}\right\vert Y_{1}^{i})$ compensates the term $\frac{\sqrt{n-i}}{\Phi
(t_{i})\sqrt{n-i-1}}\exp \left( \frac{-a\mu _{3}^{i}}{2s_{i}^{2}(n-i-1)}%
\right) ,$ where the term $\Phi (t_{i})$ comes from $\pi^{m_{i}}\left( 
\mathbf{X}_{i+1}=Y_{i+1}\right) .$ Further the product of the remaining
terms in the above approximations in\ (\ref{num approx fixed i}) and (\ref%
{PI 0}) turn to build the $1+o_{\mathfrak{P}_{n}}\left( \epsilon_{n}\left(
\log n\right) ^{2}\right) $ approximation rate, as claimed$.$ Details are
differed to the Appendix. This yields\ 
\begin{equation*}
p(\left. \mathbf{X}_{1}^{k}=Y_{1}^{k}\right\vert \mathbf{S}%
_{1}^{n}=na)=\left( 1+o_{\mathfrak{P}_{n}}\left( \epsilon _{n}\left( \log
n\right) ^{2}\right) \right) \prod_{i=0}^{k-1}g_{i}(\left.
Y_{i+1}\right\vert Y_{1}^{i})
\end{equation*}%
which closes the proof of the Theorem.
\end{proof}

\begin{remark}
When the $\mathbf{X}_{i}$'s are i.i.d. with a standard normal density, then
the result in the above approximation Theorem holds with $k=n-1$ stating
that $p(\left. \mathbf{X}_{1}^{n-1}=x_{1}^{n-1}\right\vert \mathbf{S}%
_{1}^{n}=na)=g_{a}\left( x_{1}^{n-1}\right) $ for all $x_{1}^{n-1}$ in $%
\mathbb{R}^{n-1}$. This extends to the case when they have an infinitely
divisible distribution. However formula (\ref{local approx under exact value}%
) holds true without the error term only in the gaussian case. Similar exact
formulas can be obtained for infinitely divisible distributions using (\ref%
{joint density}) making no use of tilting. Such formula is used to produce
Tables 1 and 2 in order to assess the validity of the selection rule for $k$
in the exponential case.
\end{remark}

\begin{remark}
The density in (\ref{g-i}) is a slight modification of $\pi^{m_{i}}.$ The
modification from $\pi^{m_{i}}$ to $g_{i}$ is a small shift in the location
parameter depending both on $a_{n}$ and on the skewness of $p$, and a change
in the variance : large values of $\ \mathbf{X}_{i+1}$ have smaller weigth
for large $i,$ so that the distribution of $\ \mathbf{X}_{i+1}$ tends to
concentrate around $m_{i}$ as $i$ approaches $k.$
\end{remark}

\begin{remark}
\label{Remark Edgeworth array}In the previous Theorem, as in Lemma \ref%
{Lemma max X_i under conditioning}, we use an Edgeworth expansion for the
density of the normalized sum of the $n-$th row of some triangular array of
row-wise independent r.v's with common density. Consider the i.i.d. r.v's $%
\mathbf{X}_{1},...,\mathbf{X}_{n}$ with common density $\pi^{a}(x)$ where $a$
may depend on $n$ but remains bounded$.$ The Edgeworth expansion pertaining
to the normalized density of $\mathbf{S}_{1}^{n}$ under $\pi^{a}$ can be
derived following closely the proof given for example in \cite{Feller1971},
pp 532 and followings substituting the cumulants of $p$ by those of $\pi^{a}$%
. Denote $\varphi_{a}(z)$ the characteristic function of $\pi^{a}(x).$
Clearly for any $\delta>0$ there exists $q_{a,\delta}<1$ such that $%
\left\vert \varphi _{a}(z)\right\vert <$ $q_{a,\delta}$ and since $a$ is
bounded, $\sup _{n}q_{a,\delta}<1.$ Therefore the inequality (2.5) in \cite%
{Feller1971} p533 holds. With $\psi_{n}$ defined as in \cite{Feller1971},
(2.6) holds with $\varphi$ replaced by $\varphi_{a}$ and $a$ by $s(t_{a});$
(2.9) holds, which completes the proof of the Edgeworth expansion in the
simple case. The proof goes in the same way for higher order expansions.
\end{remark}

\subsection{Sampling under the approximation}

Applications of Theorem \ref{Prop approx local cond density} in Importance
Sampling procedures\ and in Statistics require a reverse result.\ So assume
that $Y_{1}^{k}$ is a random vector generated under $G_{a}$ with density $%
g_{a}.$ Can we state that $g_{a}\left( Y_{1}^{k}\right) $ is a good
approximation for $\mathfrak{p}_{n}\left( Y_{1}^{k}\right) $? This holds
true. We state a simple Lemma in this direction.

Let $\mathfrak{R}_{n}$ and $\mathfrak{S}_{n}$ denote two p.m's on $\mathbb{R}%
^{n}$ with respective densities $\mathfrak{r}_{n}$ and $\mathfrak{s}_{n}.$

\begin{lemma}
\label{Lemma commute from p_n to g_n} Suppose that for some sequence $%
\varepsilon_{n}$ $\ $which tends to $0$ as $n$ tends to infinity%
\begin{equation}
\mathfrak{r}_{n}\left( Y_{1}^{n}\right) =\mathfrak{s}_{n}\left(
Y_{1}^{n}\right) \left( 1+o_{\mathfrak{R}_{n}}(\varepsilon_{n})\right)
\label{p_n equiv g_n under p_n}
\end{equation}
as $n$ tends to $\infty.$ Then 
\begin{equation}
\mathfrak{s}_{n}\left( Y_{1}^{n}\right) =\mathfrak{r}_{n}\left(
Y_{1}^{n}\right) \left( 1+o_{\mathfrak{S}_{n}}(\varepsilon_{n})\right) .
\label{g_n equiv p_n under g_n}
\end{equation}
\end{lemma}

\begin{proof}
Denote 
\begin{equation*}
A_{n,\varepsilon_{n}}:=\left\{ y_{1}^{n}:(1-\varepsilon_{n})\mathfrak{s}%
_{n}\left( y_{1}^{n}\right) \leq\mathfrak{r}_{n}\left( y_{1}^{n}\right) \leq%
\mathfrak{s}_{n}\left( y_{1}^{n}\right) (1+\varepsilon_{n})\right\} .
\end{equation*}
It holds for all positive $\delta$%
\begin{equation*}
\lim_{n\rightarrow\infty}\mathfrak{R}_{n}\left(
A_{n,\delta\varepsilon_{n}}\right) =1.
\end{equation*}
Write 
\begin{equation*}
\mathfrak{R}_{n}\left( A_{n,\delta\varepsilon_{n}}\right) =\int \mathbf{1}%
_{A_{n,\delta\varepsilon_{n}}}\left( y_{1}^{n}\right) \frac{\mathfrak{r}%
_{n}\left( y_{1}^{n}\right) }{\mathfrak{s}_{n}(y_{1}^{n})}\mathfrak{s}%
_{n}(y_{1}^{n})dy_{1}^{n}.
\end{equation*}
Since 
\begin{equation*}
\mathfrak{R}_{n}\left( A_{n,\delta\varepsilon_{n}}\right) \leq
(1+\delta\varepsilon_{n})\mathfrak{S}_{n}\left(
A_{n,\delta\varepsilon_{n}}\right)
\end{equation*}
it follows that 
\begin{equation*}
\lim_{n\rightarrow\infty}\mathfrak{S}_{n}\left(
A_{n,\delta\varepsilon_{n}}\right) =1,
\end{equation*}
which proves the claim.
\end{proof}

As a direct by-product of Theorem \ref{Prop approx local cond density} and
Lemma \ref{Lemma commute from p_n to g_n} we obtain

\begin{theorem}
\label{ThmdeSimulation} Assume (A), (E1,2). Then when $Y_{1}^{n}$ is
generated under the distribution $G_{a}$ it holds%
\begin{equation*}
\mathfrak{p}_{n}\left( Y_{1}^{k}\right)
=g_{a}(Y_{1}^{k})(1+o_{G_{a}}(\epsilon _{n}\left( \log n\right) ^{2}))
\end{equation*}%
with $\mathfrak{p}_{n}$ defined in (\ref{local approx under exact value}).
\end{theorem}

\section{Random walks conditioned by the mean of a function of their summands}

This section extends the above results to the case when the conditioning
event writes 
\begin{equation}
\mathbf{U}_{1}^{n}:=f\left( \mathbf{X}_{1}\right) +...+f\left( \mathbf{X}%
_{n}\right) =n\left( \sigma a+\mu \right) .  \label{cond sur f(X)}
\end{equation}%
The function $f$ is real valued, $Ef\left( \mathbf{X}\right) =\mu $ and $%
Varf\left( \mathbf{X}\right) =\sigma ^{2}.$ The characteristic function of
the random variable $f\left( \mathbf{X}\right) $ is assumed to belong to $%
L^{r}$ for some $r\geq 1.$ As previously $a$ is assumed positive. Let $p_{%
\mathbf{X}}$ denote the density of the r.v. $\mathbf{X}$\textbf{.}

Assume 
\begin{equation*}
\phi_{f}(t):=E\exp tf\left( \mathbf{X}\right) <\infty
\end{equation*}
for $t$ in a non void neighborhood of $0.$ Define the functions $%
m_{f}(t),s_{f}^{2}(t)$ and $\mu_{f,3}(t)$ as the first, second and third
derivatives of $\log\mathfrak{\phi}_{f}(t).$

Denote 
\begin{equation*}
\pi _{f}^{\alpha }(x):=\frac{\exp tf(x)}{\phi _{f}(t)}p_{\mathbf{X}}\left(
x\right)
\end{equation*}%
with $m_{f}(t)=\alpha $ and $\alpha $ belongs to the support of $P_{f}$, the
distribution of $f\left( \mathbf{X}\right) ,$ with density $p_{f}.$
Conditions on $\phi_{f}(t)$ which ensure existence and uniqueness of $t$ are referred to as \textit{steepness properties}, and are exposed in \cite{Barndorff1978}

Assume that (A) holds and the sequence $\epsilon_{n}$ satisfies (E1,2).

\subsection{Approximation of the density of the runs}

Define a density $h_{\sigma a+\mu }(y_{1}^{k})$ with c.d.f. $H_{\sigma a+\mu
}$ on $\mathbb{R}^{k}$ as follows. Set 
\begin{equation*}
h_{0}(\left. y_{1}\right\vert y_{0}):=\pi _{f}^{\sigma a+\mu }(y_{1})
\end{equation*}%
with $y_{0}$ arbitrary and for $1\leq i\leq k-1$ define $h_{i}(\left.
y_{i+1}\right\vert y_{1}^{i})$ recursively.

Set $t_{i}$ the unique solution of the equation%
\begin{equation}
m_{i}:=m_{f}(t_{i})=\frac{n}{n-i}\left( \sigma a+\mu -\frac{u_{1}^{i}%
}{n}\right)  \label{mif}
\end{equation}%
where $u_{1}^{i}:=f\left( y_{1}\right) +...+f\left( y_{i}\right) .$

Define 
\begin{equation}  \label{gif}
h_{i}(\left. y_{i+1}\right\vert y_{1}^{i})=C_{i}p_{\mathbf{X}}(y_{i+1})%
\mathfrak{n}\left( \alpha \beta +\left(\sigma a+\mu\right),\alpha ,f(y_{i+1})\right)
\end{equation}
where $C_{i}$ is a normalizing constant.\ Here 
\begin{equation}
\alpha=s_{f}^{2}(t_{i})\left( n-i-1\right)
\label{a pour f(x)}
\end{equation}
\begin{equation}
\beta =t_{i}+\frac{\mu _{f,3}\left( t_{i}\right) }{2s_{f}^{4}(t_{i})\left( n-i-1\right) }.
\label{b pour f(x)}
\end{equation}%
Set%
\begin{equation}
h_{\sigma a+\mu }\left( y_{1}^{k}\right)
:=\prod\limits_{i=0}^{k-1}h_{i}(\left. y_{i+1}\right\vert y_{1}^{i}).
\label{gasigmamu}
\end{equation}

Denote $\mathfrak{P}_{n}^{f}$ the distribution of $\mathbf{X}_{1}^{n}$
conditioned upon $\left( \mathbf{U}_{1}^{n}=n\left( \sigma a+\mu \right)
\right) $ and $\mathfrak{p}_{n}^{f}$ its density when restricted on $\mathbb{%
R}^{k};$ therefore%
\begin{equation}
\mathfrak{p}_{n}^{f}\left( \mathbf{X}_{1}^{k}=Y_{1}^{k}\right) :=p(\left. 
\mathbf{X}_{1}^{k}=Y_{1}^{k}\right\vert \mathbf{U}_{1}^{n}=n\left( \sigma
a+\mu \right) ).  \label{p_n^f}
\end{equation}

\begin{theorem}
\label{Prop approx sous f(x)=a} Assume (A) and (E1,2) . Then (i) 
\begin{equation*}
\mathfrak{p}_{n}^{f}\left( \mathbf{X}_{1}^{k}=Y_{1}^{k}\right) =h_{\sigma
a+\mu }(Y_{1}^{k})(1+o_{\mathfrak{P}_{n}^{f}}(\epsilon _{n}\left( \log
n\right) ^{2}))
\end{equation*}%
and (ii) 
\begin{equation*}
\mathfrak{p}_{n}^{f}\left( \mathbf{X}_{1}^{k}=Y_{1}^{k}\right) =h_{\sigma
a+\mu }(Y_{1}^{k})(1+o_{H_{\sigma a+\mu }}(\epsilon _{n}\left( \log n\right)
^{2})).
\end{equation*}
\end{theorem}

\begin{proof}
We only sketch the initial step of the proof of (i), which rapidly follows
the same track as that in Theorem \ref{Prop approx local cond density}.
Denote $U_{i}^{j}:=f(Y_{i})+...+f(Y_{j}).$

As in the proof of \ Theorem \ref{Prop approx local cond density} evaluate 
\begin{align*}
p\left( \left. \mathbf{X}_{i+1}=Y_{i+1}\right\vert \mathbf{U}%
_{i+1}^{n}=n\left( \sigma a+\mu \right) -U_{1}^{i}\right) \\
& =p\left( \mathbf{X}_{i+1}=Y_{i+1}\right) \frac{p\left( \mathbf{U}%
_{i+2}^{n}=n\left( \sigma a+\mu \right) -U_{1}^{i+1}\right) }{p\left( 
\mathbf{U}_{i+1}^{n}=n\left( \sigma a+\mu \right) -U_{1}^{i}\right) }.
\end{align*}

Use the tilting invariance under $\mathfrak{\pi }_{f}^{m_{i}}$ leading to 
\begin{align*}
p\left( \left. \mathbf{X}_{i+1}=Y_{i+1}\right\vert \mathbf{U}%
_{i+1}^{n}=n\left( \sigma a+\mu \right) -U_{1}^{i}\right) \\
& =\pi_{f}^{m_{i}}\left( \mathbf{X}_{i+1}=Y_{i+1} \right) \frac{\pi
_{f}^{m_{i}}\left( \mathbf{U}_{i+2}^{n}=n\left( \sigma a+\mu \right)
-U_{1}^{i+1}\right) }{\pi _{f}^{m_{i}}\left( \mathbf{T}_{i+1}^{n}=n\left(
\sigma a+\mu \right) -U_{1}^{i}\right) } \\
& =p\left( \mathbf{X}_{i+1}=Y_{i+1}\right) \frac{e^{t_{i}f\left(
Y_{i+1}\right) }}{\phi _{f}(t_{i})}\frac{\pi _{f}^{m_{i}}\left( \mathbf{U}%
_{i+2}^{n}=n\left( \sigma a+\mu \right) -U_{1}^{i+1}\right) }{\pi
_{f}^{m_{i}}\left( \mathbf{U}_{i+1}^{n}=n\left( \sigma a+\mu \right)
-U_{1}^{i}\right) }
\end{align*}%
and proceed through the Edgeworth expansions in the above expression,
following verbatim the proof of Theorem \ref{Prop approx local cond density}%
. We omit details. The proof of (ii) follows from Lemma \ref{Lemma commute
from p_n to g_n}
\end{proof}

\subsection{How far is the approximation valid?}

This section provides a rule leading to an effective choice of the crucial
parameter $k$ in order to achieve a given accuracy bound for the relative
error. The generic r.v. $\mathbf{X}$ has density $p_{\mathbf{X}}$ and $%
f\left( \mathbf{X}\right) $ has mean $\mu $ and variance $\sigma ^{2}.$ The
density $\mathfrak{p}_{n}^{f}$ is defined in (\ref{p_n^f}). The accuracy of
the approximation is measured through%
\begin{equation*}
ERE(k):=E_{h_{\sigma a+\mu }}1_{D_{k}}\left( Y_{1}^{k}\right) \frac{%
\mathfrak{p}_{n}^{f}\left( Y_{1}^{k}\right) -h_{\sigma a+\mu }\left(
Y_{1}^{k}\right) }{\mathfrak{p}_{n}^{f}\left( Y_{1}^{k}\right) }
\end{equation*}%
and 
\begin{equation}
VRE(k):=Var_{h_{\sigma a+\mu }}1_{D_{k}}\left( Y_{1}^{k}\right) \frac{%
\mathfrak{p}_{n}^{f}\left( Y_{1}^{k}\right) -h_{\sigma a+\mu }\left(
Y_{1}^{k}\right) }{\mathfrak{p}_{n}^{f}\left( Y_{1}^{k}\right) }  \label{RE}
\end{equation}%
respectively the expectation and the variance of the relative error of the
approximating scheme when evaluated on $D_{k}$ , the subset of $\mathbb{R}%
^{k}$ \ where $\left\vert h_{\sigma a+\mu }(Y_{1}^{k})/\mathfrak{p}%
_{n}^{f}\left( Y_{1}^{k}\right) -1\right\vert <\delta _{n}$ with $\epsilon
_{n}\left( \log n\right) ^{2}/\delta _{n}\rightarrow 0$ and $\delta
_{n}\rightarrow 0;$ therefore $H_{\sigma a+\mu }\left( D_{k}\right)
\rightarrow 1.$The r.v$^{\prime }$s $Y_{1}^{k}$ are sampled under $h_{\sigma
a+\mu }.$ Note that the density $\mathfrak{p}_{n}^{f}$ is usually unknown.
The argument is somehow heuristic and unformal; nevertheless the rule is
simple to implement and provides good results. We assume that the set $D_{k}$
$\ $can be substituted by $\mathbb{R}^{k}$ in the above formulas, therefore
assuming that the relative error has bounded variance, which would require
quite a lot of work to be proved under appropriate conditions, but which
seems to hold, at least in all cases considered by the authors. We keep the
above notation omitting therefore any reference to $D_{k}$ .

Consider a two-sigma confidence bound for the relative accuracy for a given $%
k$, defining%
\begin{equation*}
CI(k):=\left[ ERE(k)-2\sqrt{VRE(k)},ERE(k)+2\sqrt{VRE(k)}\right] .
\end{equation*}%
\bigskip

Let $\delta $ denote an acceptance level for the relative accuracy. Accept $%
k $ until $\delta $ belongs to $CI(k).$ For such $k$ the relative accuracy
is certified up to the level $5\%$ roughly.

The calculation of $VRE(k)$ and $ERE(k)$ should be done as follows.

Write 
\begin{align}
VRE(k)^{2}& =E_{p_{\mathbf{X}}}\left( \frac{h_{\sigma a+\mu }^{3}\left(
Y_{1}^{k}\right) }{\mathfrak{p}_{n}^{f}\left( Y_{1}^{k}\right) ^{2}p_{\mathbf{X}}\left( Y_{1}^{k}\right)}\right) \\
& -E_{p_{\mathbf{X}}}\left( \frac{h_{\sigma a+\mu }^{2}\left( Y_{1}^{k}\right) }{\mathfrak{p}_{n}^{f}\left( Y_{1}^{k}\right)p_{\mathbf{X}}\left( Y_{1}^{k}\right)}\right) ^{2} \\
& =:A-B^{2}.
\end{align}%
where $\pi ^{a}$ is defined in (\ref{tilted density}). By Bayes formula

\begin{equation} \label{Jensen_dans_k_limite}
\mathfrak{p}_{n}^{f}\left( Y_{1}^{k}\right) =p_{\mathbf{X}}\left(
Y_{1}^{k}\right) \frac{np_{\mathbf{U}_{k+1}^{n}/(n-k)}\left(
m_{f}(t_{k})\right) }{\left( n-k\right) p_{\mathbf{U}_{1}^{n}/n}\left( \sigma a+\mu \right) }. 
\end{equation}%
The following Lemma holds; see \cite{Jensen1995} and \cite{Richter1957}.

\begin{lemma} \label{LemmaJensen} 
Let $\mathbf{X}_{1},...,\mathbf{X}_{n}$ be i.i.d. random variables with common density $p$ on $\mathbb{R}$ and satisfying the Cramer conditions with m.g.f. $\phi$. Then with $\left( \log\phi\right)^{\prime }(t)=u,$ 
\begin{equation*}
p_{\mathbf{S}_{1}^{n}/n}\left( u\right) =\frac{\sqrt{n}\phi^{n}(t)\exp -ntu}{%
s(t)\sqrt{2\pi}}\left( 1+o(1)\right)
\end{equation*}
in the range of the large or moderate deviations, i.e. when $|u|\sqrt {n}%
\rightarrow\infty$ and $|u|$ is bounded from above.
\end{lemma}

Introduce%
\begin{equation*}
D:=\left[ \frac{\pi _{f}^{a}(a)}{p_{\mathbf{X}}(a)}\right] ^{n}
\end{equation*}%
and%
\begin{equation*}
N:=\left[ \frac{\pi _{f}^{m_{k}}\left( m_{k}\right) }{p_{\mathbf{X}}\left(
m_{k}\right) }\right] ^{\left( n-k\right) }
\end{equation*}%
with $m_{k}$ defined in (\ref{mif})$.$ Define $t$ by $m_{f}(t)=\sigma a+\mu .$ By (\ref{Jensen_dans_k_limite}) and Lemma \ref{LemmaJensen} it holds
\begin{equation*}
\mathfrak{p}_{n}^{f}\left( Y_{1}^{k}\right) =\sqrt{\frac{n}{n-k}}p_{\mathbf{X%
}}\left( Y_{1}^{k}\right) \frac{N}{D}\frac{s_{f}(t)}{%
s_{f}(t_{k})}\left( 1+o_{P_{\mathbf{X}}}(1)\right) .
\end{equation*}%
The approximation of $A$ is obtained through Monte Carlo simulation. Define%
\begin{equation}
A\left( Y_{1}^{k}\right) :=\frac{n-k}{n}\left( \frac{h_{\sigma a+\mu }\left(
Y_{1}^{k}\right) }{p_{\mathbf{X}}\left( Y_{1}^{k}\right) }\right) ^{3}\left( 
\frac{D}{N}\right) ^{2}\frac{s_{f}^{2}(t_{k})}{s_{f}^{2}(t)}
\label{A(l)}
\end{equation}
and simulate $L$ i.i.d. samples $Y_{1}^{k}(l)$ , each one made of $k$ i.i.d.
replications under $p_{\mathbf{X}}$; set 
\begin{equation*}
\widehat{A}:=\frac{1}{L}\sum_{l=1}^{L}A\left( Y_{1}^{k}(l)\right) .
\end{equation*}%
We use the same approximation for $B.$ Define%
\begin{equation}
B\left( Y_{1}^{k}\right) :=\sqrt{\frac{n-k}{n}}\left( \frac{h_{\sigma a+\mu
}\left( Y_{1}^{k}\right) }{p_{\mathbf{X}}\left( Y_{1}^{k}\right) }\right)
^{2}\left( \frac{D}{N}\right) \frac{s_{f}(t_{k})}{s_{f}(t)}
\label{B(l)}
\end{equation}
and 
\begin{equation*}
\widehat{B}:=\frac{1}{L}\sum_{l=1}^{L}B\left( Y_{1}^{k}(l)\right) 
\end{equation*}%
with the same $Y_{1}^{k}(l)^{\prime }s$ as above.

Set 
\begin{equation}
\overline{VRE}(k):=\widehat{A}-\left( \widehat{B}\right) ^{2}  \label{REbar}
\end{equation}
which is a fair approximation of $VRE(k).$

The curve $k\rightarrow \overline{ERE}(k)$ is a proxy for 
\begin{equation*}
ERE(k):=E_{h_{\sigma a+\mu }}\frac{\mathfrak{p}_{n}^{f}\left(
Y_{1}^{k}\right) -h_{\sigma a+\mu }\left( Y_{1}^{k}\right) }{\mathfrak{p}%
_{n}^{f}\left( Y_{1}^{k}\right) }
\end{equation*}%
and is obtained through 
\begin{equation*}
\overline{ERE}(k):=1-\widehat{B}.
\end{equation*}%
A proxy of $CI(k)$ can now be defined through 
\begin{equation}
\overline{CI}(k):=\left[ \overline{ERE}(k)-2\sqrt{\overline{VRE}(k)},%
\overline{ERE}(k)+2\sqrt{\overline{VRE}(k)}\right] .  \label{CIbarre}
\end{equation}

We now check the validity of the just above approximation, comparing $%
\overline{CI}(k)$ with $CI(k)$ on a toy case.

Consider $f(x)=x.$ The case when $p$ is a centered exponential distribution
with variance $1$ allows for an explicity evaluation of $CI(k)$ making no
use of Lemma \ref{LemmaJensen}. The conditional density $\mathfrak{p}_{n}$
is calculated analytically, the density $g_{a}$ is obtained through (\ref%
{g_s}), hence providing a benchmark for our proposal. The terms $\widehat{A}$
and $\widehat{B}$ are obtained by Monte Carlo simulation following the
algorithm presented hereunder. Tables 1,2 and 3,4 show the increase in $%
\delta $ w.r.t. $k$ in the moderate deviation range, with $a$ such that $%
P\left( \mathbf{S}_{1}^{n}>na\right) \simeq 10^{-2}.$ In Table 5,6 and
7,8, $a$ is such that $P\left( \mathbf{S}_{1}^{n}>na\right) \simeq 10^{-8}$
corresponding to a large deviation case. We have considered two cases, when $%
n=100$ and when $n=1000.$
These tables show that the approximation scheme is quite accurate, since the relative error is fairly small even when approximating event is in spaces with very high dimension. Also they show that $\bar{ERE}$ et $\bar{CI}$ provide good tools for the assessing the value of $k.$

\begin{figure}[h!t]
\begin{minipage}[b]{0.50\linewidth}
      \centering \includegraphics*[scale=0.25]{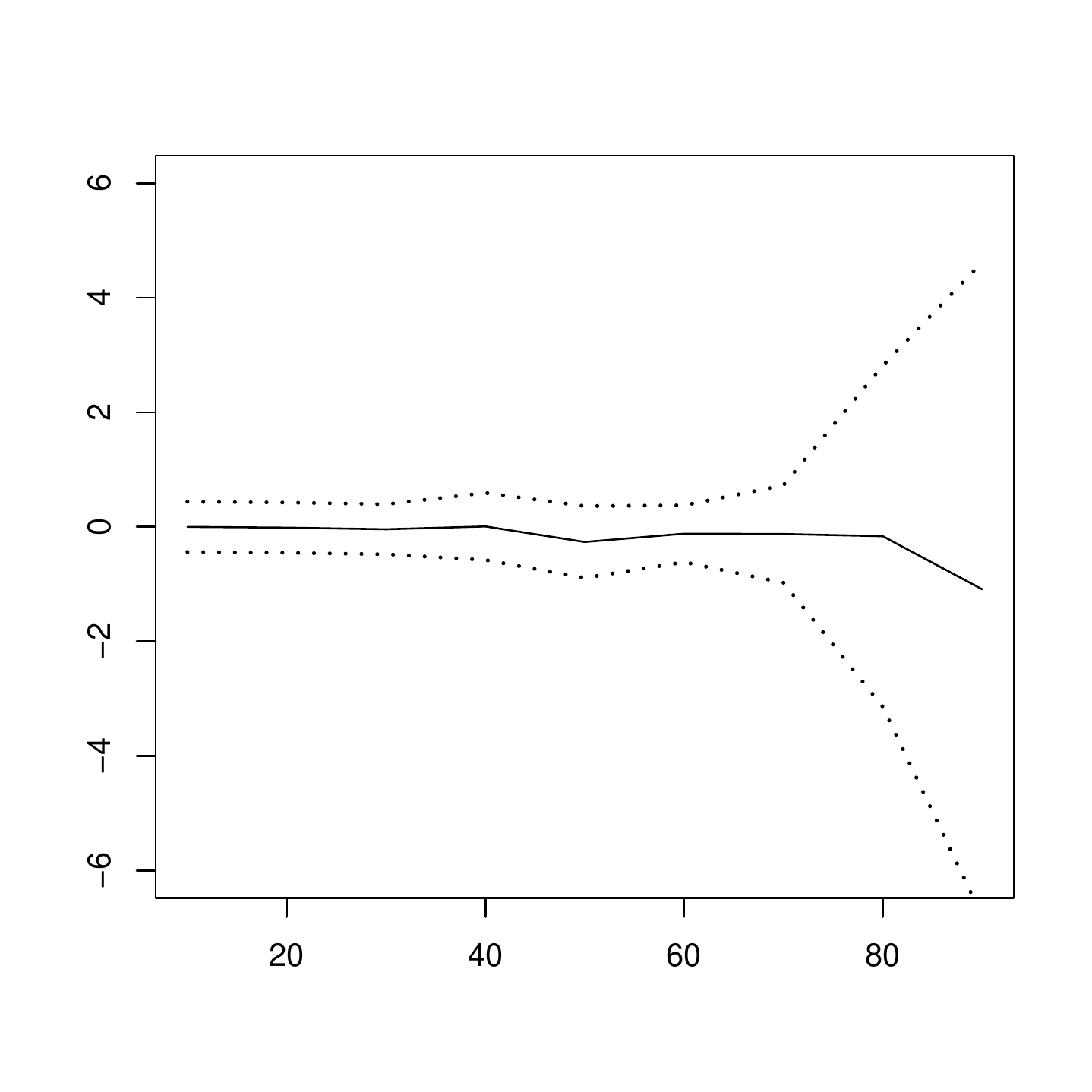}
      \caption{$
\overline{CI}(k)$ for n=100\\ and $P\left(  \mathbf{S}_{1}^{n}>na\right)  \simeq10^{-2}.$}
   \end{minipage}\hfill 
\begin{minipage}[b]{0.50\linewidth}   
      \centering \includegraphics*[scale=0.25]{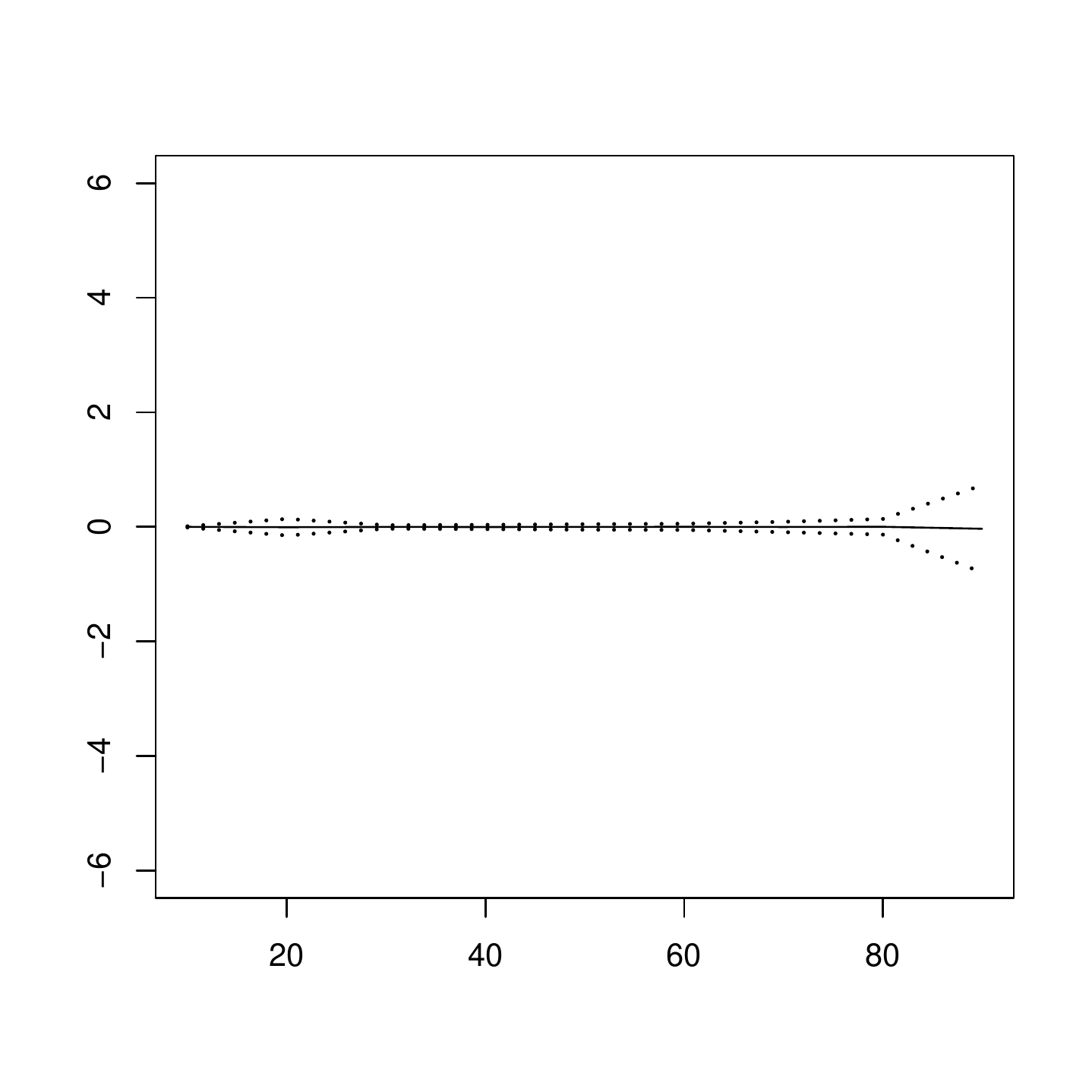}
      \caption{$CI(k)$ for n=100\\ and $P\left(  \mathbf{S}_{1}^{n}>na\right)  \simeq10^{-2}.$}
   \end{minipage}
\end{figure}

\newpage

\begin{figure}[h!t]
\begin{minipage}[b]{0.50\linewidth}
      \centering \includegraphics*[scale=0.25]{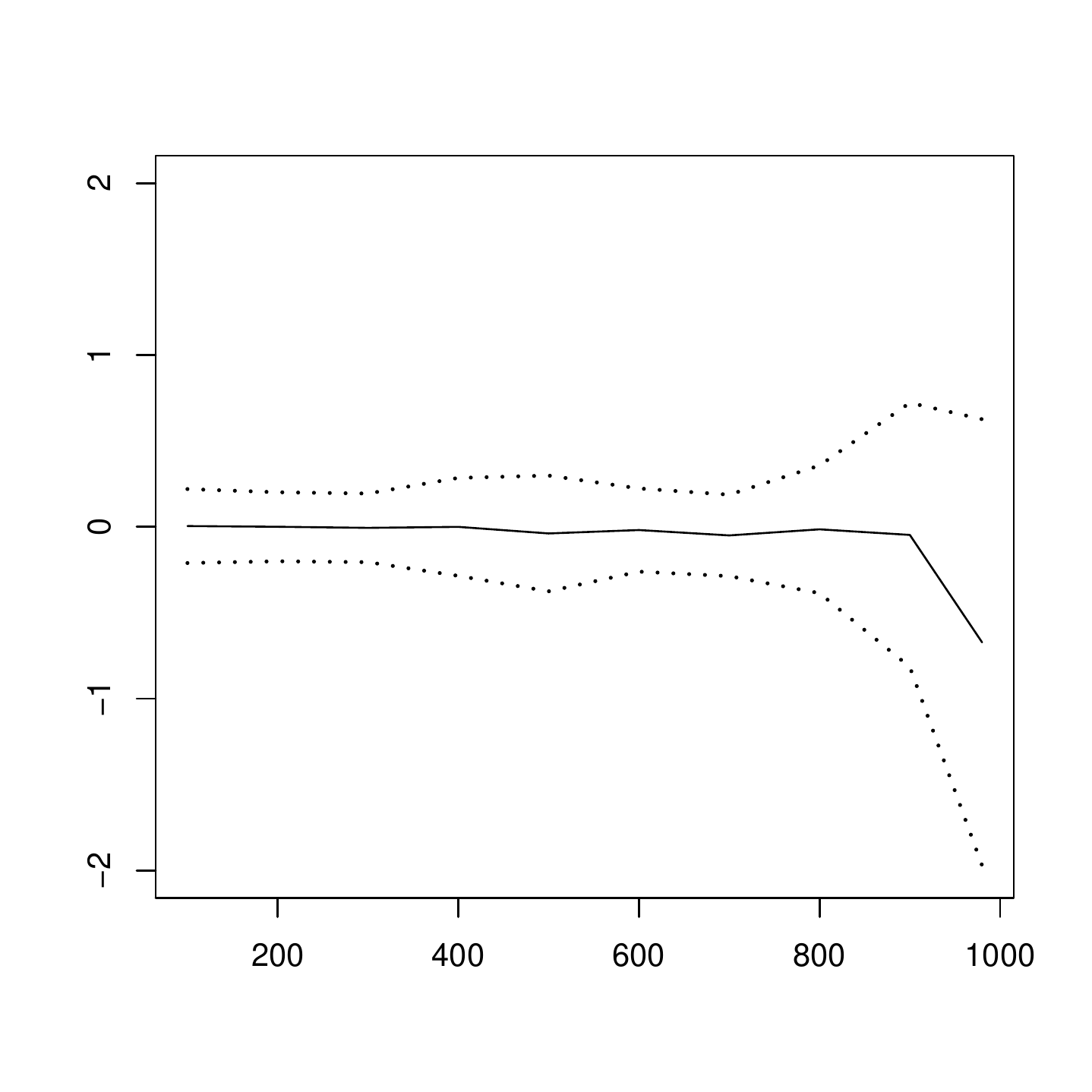}
      \caption{$
\overline{CI}(k)$ for n=1000\\ and $P\left(  \mathbf{S}_{1}^{n}>na\right)  \simeq10^{-2}.$}
   \end{minipage}\hfill 
\begin{minipage}[b]{0.50\linewidth}   
      \centering \includegraphics*[scale=0.25]{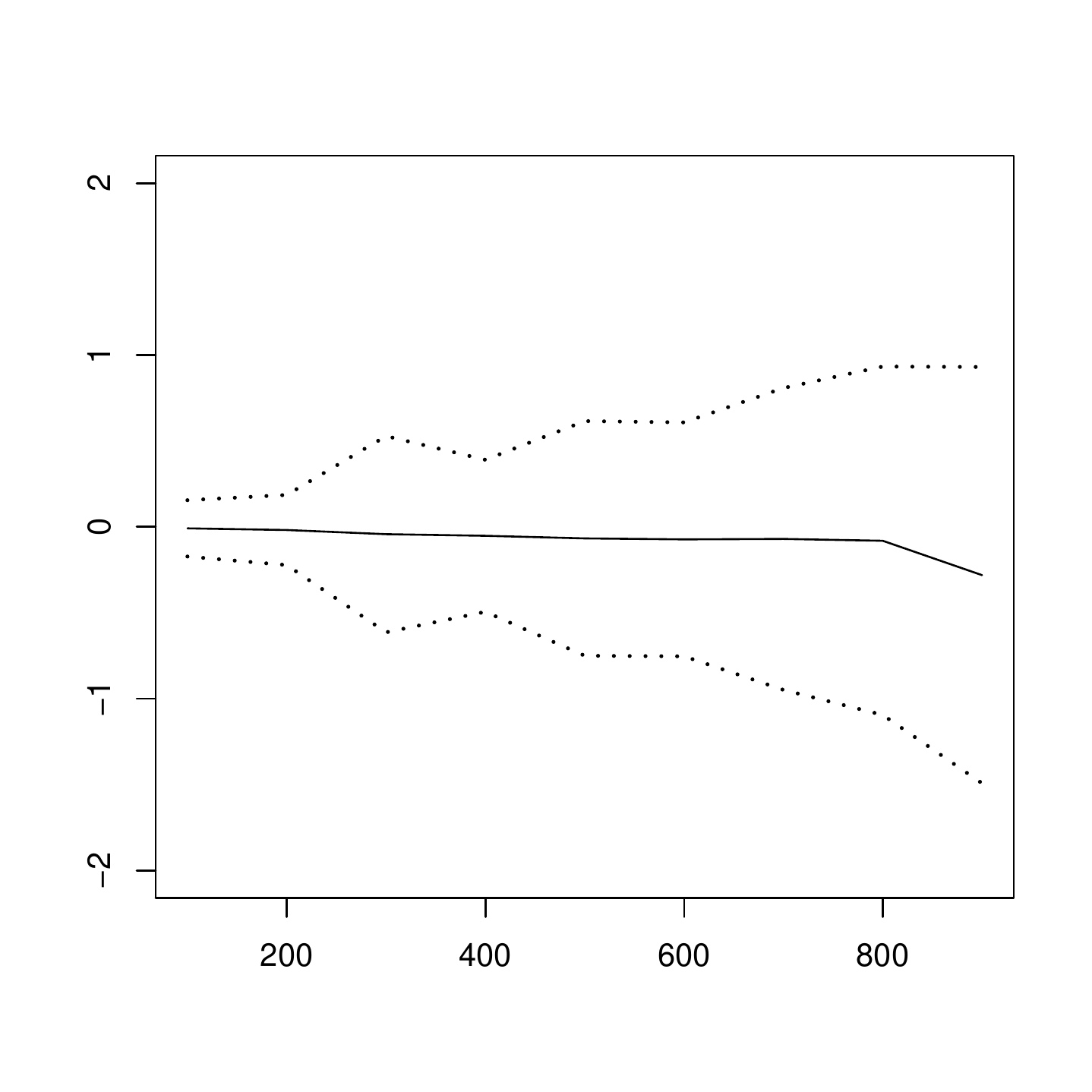}
      \caption{$CI(k)$ for n=1000\\ and $P\left(  \mathbf{S}_{1}^{n}>na\right)  \simeq10^{-2}.$}
   \end{minipage}
\end{figure}

\begin{figure}[h!t]
\begin{minipage}[b]{0.50\linewidth}
      \centering \includegraphics*[scale=0.25]{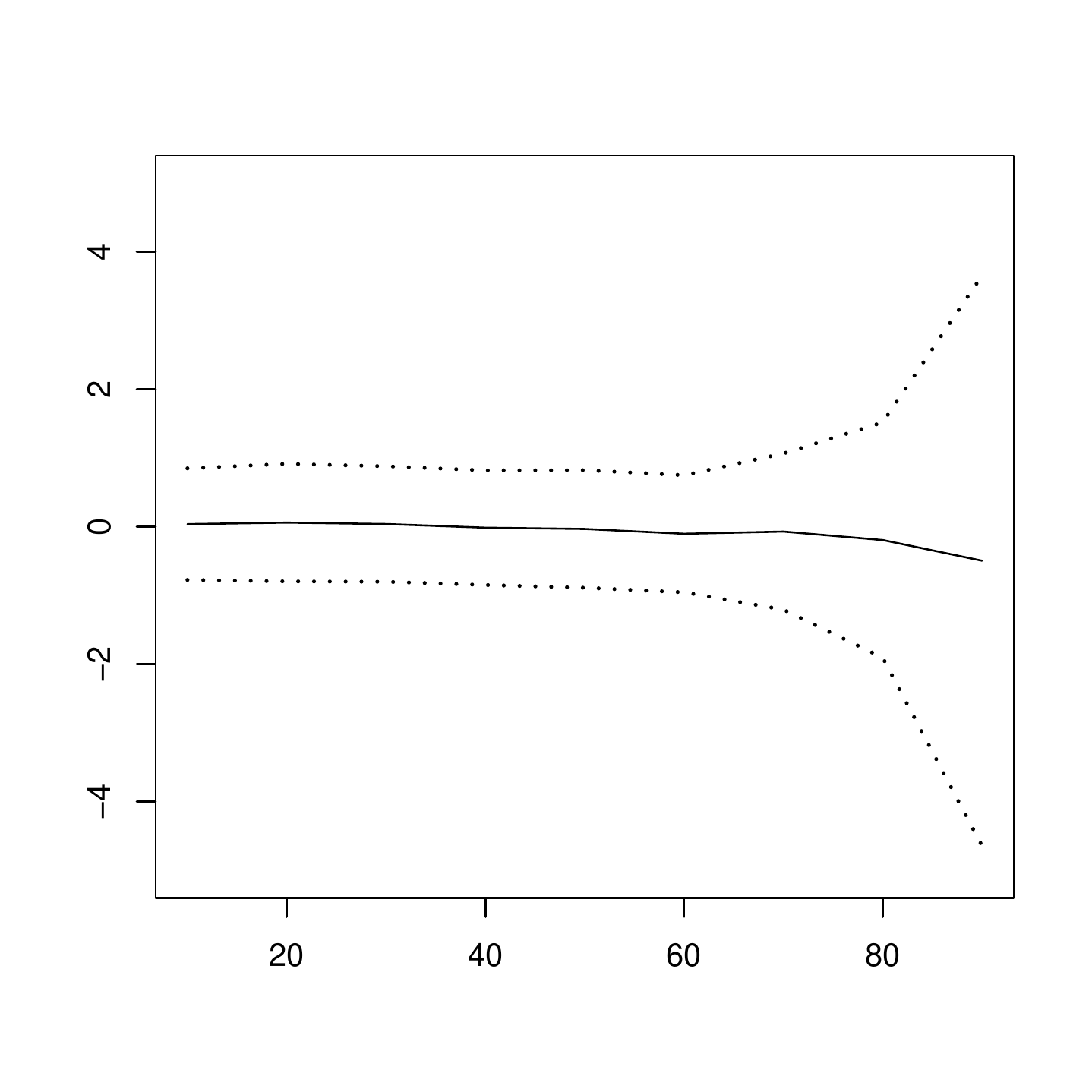}
      \caption{$
\overline{CI}(k)$ for n=100\\ and $P\left(  \mathbf{S}_{1}^{n}>na\right)  \simeq10^{-8}.$}
   \end{minipage}\hfill 
\begin{minipage}[b]{0.50\linewidth}   
      \centering \includegraphics*[scale=0.25]{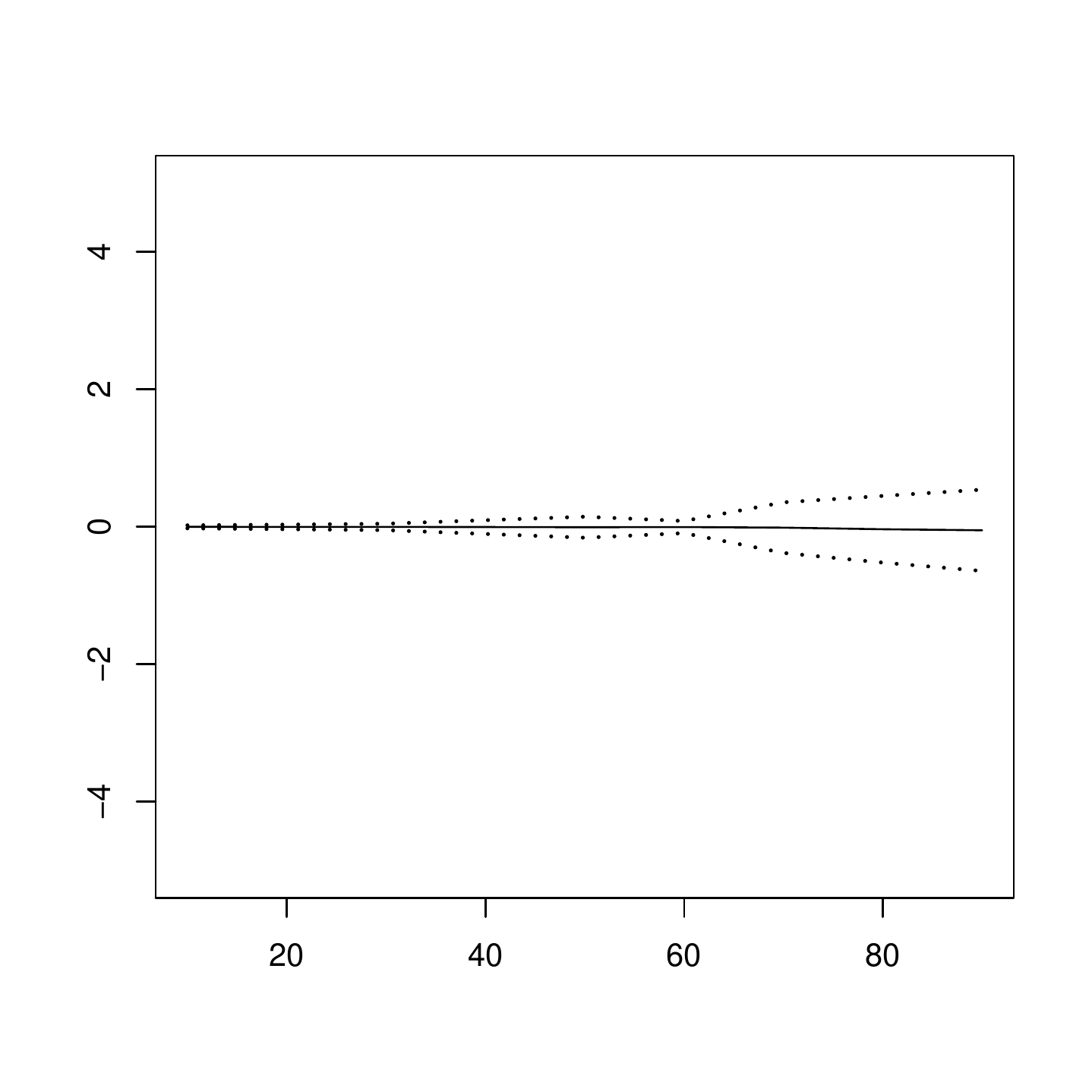}
      \caption{$CI(k)$ for n=100\\ and $P\left(  \mathbf{S}_{1}^{n}>na\right)  \simeq10^{-8}.$}
   \end{minipage}
\end{figure}

\begin{figure}[h!t]
\begin{minipage}[b]{0.50\linewidth}
      \centering \includegraphics*[scale=0.25]{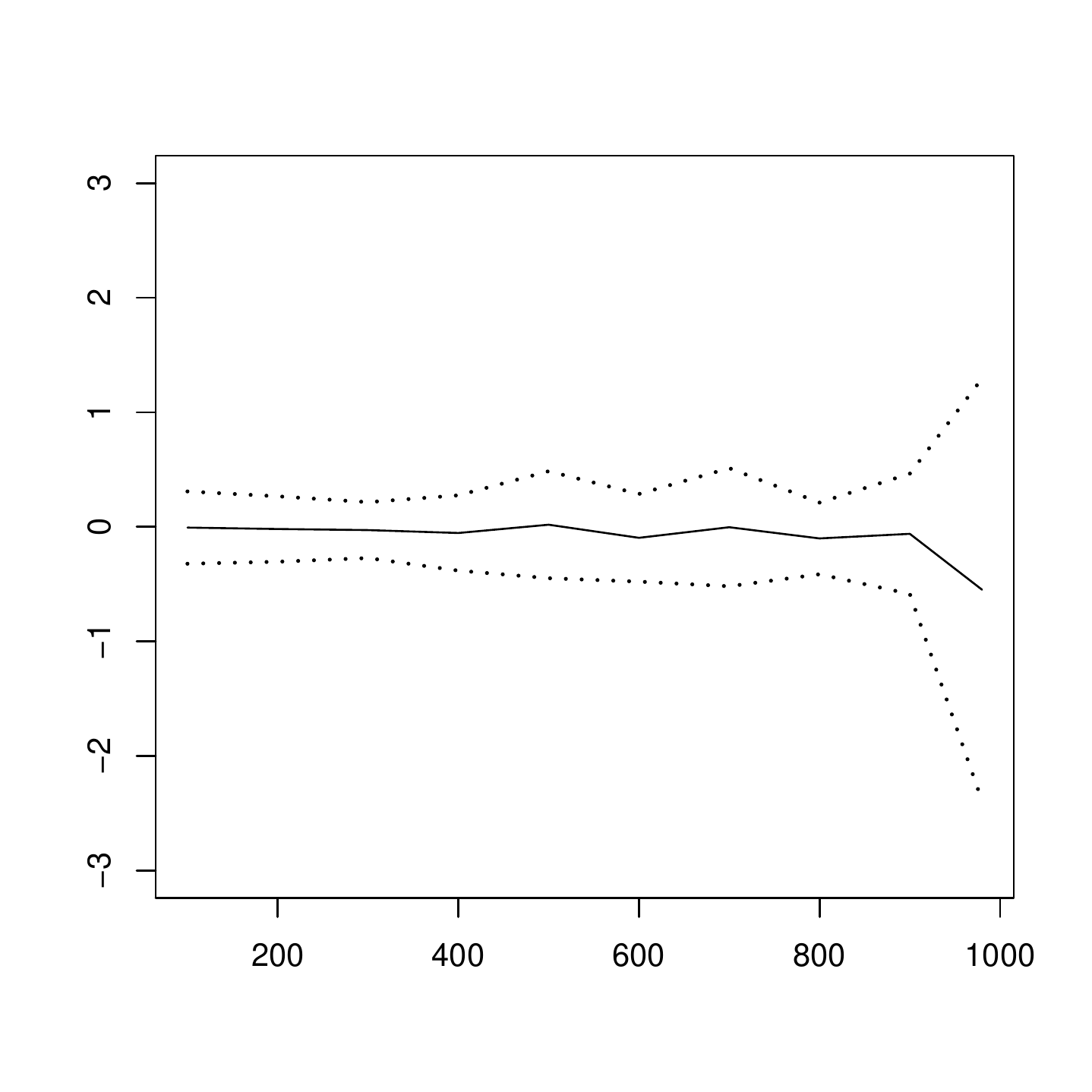}
      \caption{$
\overline{CI}(k)$ for n=1000\\ and $P\left(  \mathbf{S}_{1}^{n}>na\right)  \simeq10^{-8}.$}
   \end{minipage}\hfill 
\begin{minipage}[b]{0.50\linewidth}   
      \centering \includegraphics*[scale=0.25]{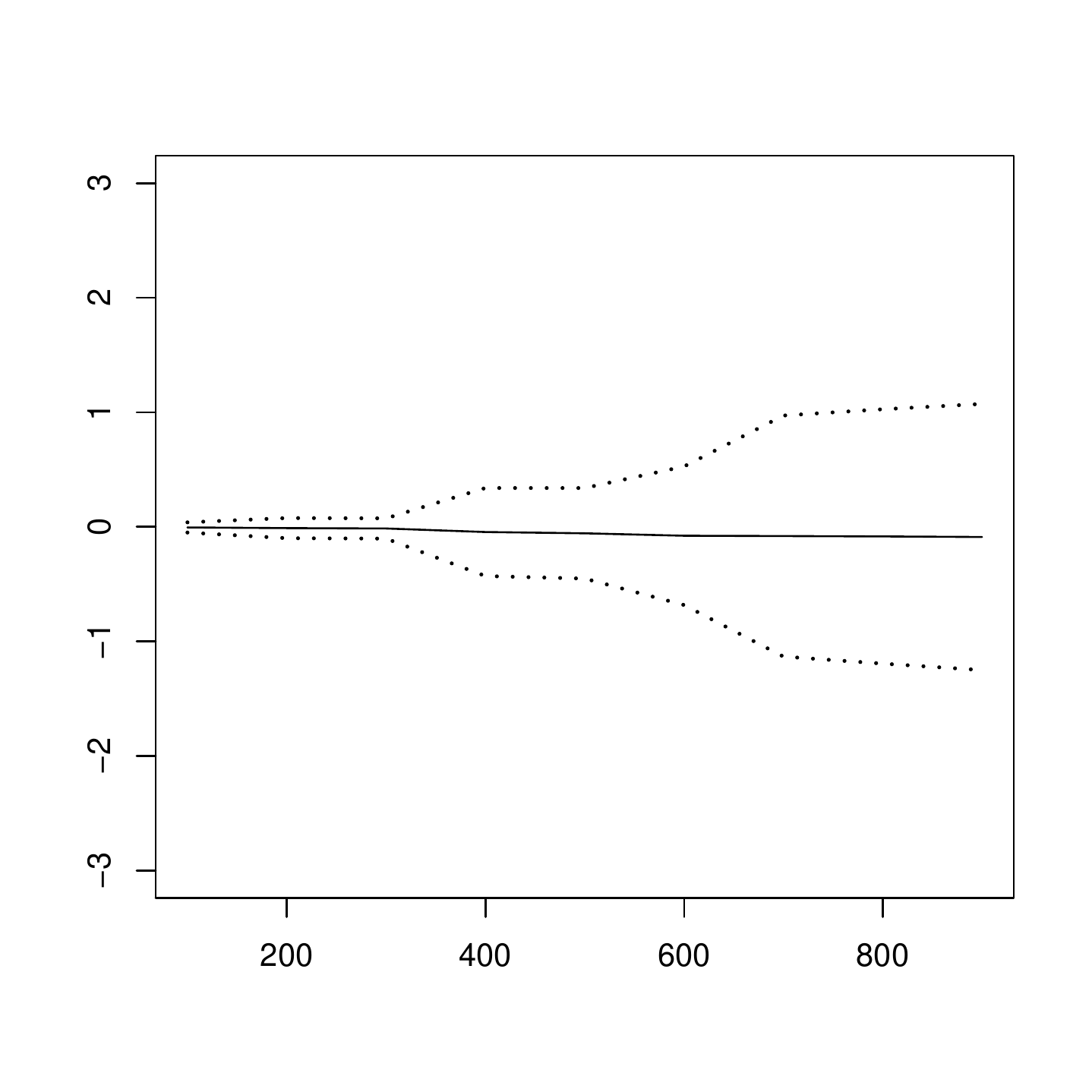}
      \caption{$CI(k)$ for n=1000\\ and $P\left(  \mathbf{S}_{1}^{n}>na\right)  \simeq10^{-8}.$}
   \end{minipage}
\end{figure}

\newpage

We present a series of algorithms which produces the curve $k\rightarrow 
\overline{RE}(k)$ in the case when $f\left( \mathbf{X}\right) $ has
expectation $\mu $ and variance $\sigma ^{2}.$

\begin{center}
\begin{tabular}{|l|l|}
\hline
\textsf{Algorithm 1 Evaluates the function }$h_{\sigma a+\mu }$ &  \\ \hline
\textbf{INPUT} &  \\ 
\qquad \qquad \qquad vector $x_{1}^{k}$, integer $n$, density $p_{\mathbf{X}%
} $, level $a$ &  \\ 
\textbf{OUTPUT} &  \\ 
\qquad \qquad \qquad \textbf{\ }$h_{\sigma a+\mu }\left( x_{1}^{k}\right) $
&  \\ 
\textbf{INITIALIZATION} &  \\ 
$\qquad \qquad \qquad t_{0}:=m_{f}^{-1}\left( \sigma a+\mu \right) $
&  \\ 
$\qquad \qquad \qquad h_{0}(\left. x_{1}\right\vert x_{1}^{0}):=\pi
_{f}^{\sigma a+\mu }(x_{1})$ &  \\ 
$\qquad \qquad \qquad \Sigma _{1}^{1}:=x_{1}$ &  \\ 
\textbf{PROCEDURE} &  \\ 
\qquad \qquad \qquad \textbf{For }$i$ from $1$ to $k-1$ &  \\ 
\qquad \qquad \qquad \qquad $m_{i}:=$(\ref{mif}) &  \\ 
\qquad \qquad \qquad \qquad $t_{i}:=m_{f}^{-1}(m_{i})$ & 
\\ 
\qquad \qquad \qquad \qquad $\alpha :=$(\ref{a pour f(x)}) &  \\ 
\qquad \qquad \qquad \qquad $\beta :=$(\ref{b pour f(x)}) &  \\ 
\qquad \qquad \qquad \qquad Calculate $C_{i}$ in (\ref{gif}) through
MonteCarlo &  \\ 
\qquad \qquad \qquad \qquad $h_{i}(\left. x_{i+1}\right\vert x_{1}^{i}):=$(%
\ref{gif}) &  \\ 
\qquad \qquad \qquad \textbf{endFor} &  \\ 
\qquad \qquad \qquad Compute &  \\ 
\qquad \qquad \qquad \qquad \textbf{\ }$h_{\sigma a+\mu }\left(
x_{1}^{k}\right) :=$(\ref{gasigmamu}) &  \\ 
\qquad \qquad \qquad Return $h_{\sigma a+\mu }\left( x_{1}^{k}\right) $ & 
\\ \hline
\end{tabular}
\end{center}

\begin{remark}
Solving $t_{i}:=m_{f}^{-1}(m_{i})$ might be difficult, even through a
Newton Raphson technique and time consuming in simple cases. It may happen
that the reciprocal function of $m_{f}$ is at hand, as is assumed in Dupuis
and Wang \cite{DupuisWang2004}, but even in such current situation as the
Weibull distribution and $f(x)=x$, such is not the case. An alternative
computation is presented in Algorithm 1', following an expansion in $%
m_{f}(t_{i})$, which is a good update since $s_{f}(t_{i-1})$
is stable around $varf\left( \mathbf{X}\right) $ in the case when $a$ tends
to $Ef\left( \mathbf{X}\right) $ or to $s_{f}^{2}(t)$ when $a$ tends to $0$
or for fixed $a$, as follows from a variant of Lemma \ref{LemmaMaxm_in}$.$
\end{remark}

\begin{center}
\begin{tabular}{|l|l|}
\hline
\textsf{Algorithm 1'} &  \\ \hline
\ Similar to Algorithm 1 with $t_{i}:=m_{f}^{-1}(m_{i})$ substituted
by &  \\ 
\qquad \qquad \qquad \qquad $t_{i}:=t_{i-1}+\frac{m_{f}\left(t_{i-1}\right)-x_{i}}{\left( n-i\right)
s_{f}^{2}\left(t_{i-1}\right) }.$ &  \\ \hline
\end{tabular}

\begin{tabular}{|l|l|}
\hline
\textsf{Algorithm 2 : Calculates }$k_{\delta }$ &  \\ \hline
\textbf{INPUT \ } &  \\ 
\qquad \qquad \qquad\ density $p_{\mathbf{X}}$, level $a$, efficiency $%
\delta ,$ integer $n,$ integer $L$ &  \\ 
\textbf{OUTPUT} &  \\ 
\qquad \qquad \qquad $k_{\delta }$ &  \\ 
\textbf{INITIALIZE} &  \\ 
\qquad \qquad \qquad $k=1$ &  \\ 
\textbf{PROCEDURE} &  \\ 
\qquad \qquad \qquad \textbf{Do} &  \\ 
\qquad \qquad \qquad \qquad \textbf{For }$l$ from $1$ to $L$ &  \\ 
\qquad \qquad \qquad \qquad \qquad Simulate $Y_{1}^{k}(l)$ i.i.d. with
density $p_{\mathbf{X}}$ &  \\ 
\qquad \qquad \qquad \qquad $\qquad A\left( Y_{1}^{k}(l)\right) :=$(\ref%
{A(l)}) using Algorithm 1 &  \\ 
\qquad \qquad \qquad \qquad $\qquad B\left( Y_{1}^{k}(l)\right) :=$(\ref%
{B(l)}) using Algorithm 1 &  \\ 
\qquad \qquad \qquad \qquad \textbf{endFor} &  \\ 
\qquad \qquad \qquad \qquad $\qquad $Calculate $\overline{CI}(k):=$(\ref%
{CIbarre}) &  \\ 
\qquad \qquad \qquad \qquad \textbf{\ \ \ \ \ \ \ }$k:=k+1$ &  \\ 
\qquad \qquad \qquad \textbf{While }$\delta \notin \overline{CI}(k)$ &  \\ 
\qquad \qquad \qquad \textbf{endDo} &  \\ 
\qquad \qquad \qquad Return $k_{\delta }$\bigskip $:=k$ &  \\ \hline
\end{tabular}
\end{center}

\section{Simulation of typical paths of a random walk under a point
conditioning}

By Theorem \ref{Prop approx sous f(x)=a} (ii), $h_{\sigma a+\mu }$ and the
density of $\mathbf{X}_{1}^{k}$ under $\left( \mathbf{U}_{1}^{n}=n\left(
\sigma a+\mu \right) \right) $ get closer and closer on a family of subsets
of $\mathbb{R}^{k}$ which bear the typical paths of the random walk under
the conditioning $(\mathbf{U}_{1}^{n}=n\left( \sigma a+\mu \right) )$ with
probability going to $1$ as $n$ increases$.$ By Lemma \ref{Lemma commute
from p_n to g_n} large sets under $\mathfrak{P}_{n}^{f}$ are also large sets
under $H_{\sigma a+\mu }$. It follows that longs runs of typical paths under 
$\mathfrak{p}_{n}^{f}$ defined in (\ref{p_n^f}) can be simulated as typical
paths under $H_{\sigma a+\mu }$ at least for large $n.$

\begin{center}
\begin{tabular}{|l|l|}
\hline
\textsf{Algorithm 3 : Simulates a sample }$Y_{1}^{k}$\textsf{\ with density }
\textbf{\ }$h_{\sigma a+\mu }$ &  \\ \hline
\textbf{INPUT} \qquad \qquad \qquad integer $n$, density $p_{\mathbf{X}}$,
level $a$ , accuracy $\delta $ &  \\ 
\textbf{OUTPUT} &  \\ 
\qquad \qquad \qquad Vector\textbf{\ }$Y_{1}^{k}$ &  \\ 
\textbf{INITIALIZATION} &  \\ 
\qquad \qquad \qquad Set $k$:$=k_{\delta }$ with Algorithm 2 &  \\ 
$\qquad \qquad \qquad t_{0}:=m_{f}^{-1}\left( \sigma a+\mu \right) $
&  \\ 
\textbf{PROCEDURE} &  \\ 
\qquad \qquad \qquad Simulate $Y_{1}$ with density $\pi _{f}^{\sigma a+\mu }$
&  \\ 
\qquad \qquad \qquad $\Sigma _{1}^{1}:=Y_{1}$ &  \\ 
\qquad \qquad \qquad \textbf{For }$i$ from $1$ to $k-1$ &  \\ 
\qquad \qquad \qquad \qquad $m_{i}:=$(\ref{mif}) &  \\ 
\qquad \qquad \qquad \qquad $t_{i}:=m_{f}^{-1}(m_{i})$ &  \\ 
\qquad \qquad \qquad \qquad $\alpha :=$(\ref{a pour f(x)}) &  \\ 
\qquad \qquad \qquad \qquad $\beta :=$(\ref{b pour f(x)}) &  \\ 
\qquad \qquad \qquad \qquad Simulate $Y_{i+1}$ with density $h_{i}(\left.
y_{i+1}\right\vert y_{1}^{i})$ &  \\ 
\qquad \qquad \qquad \qquad $\Sigma _{1}^{i+1}:=\Sigma _{1}^{i}+Y_{i+1}$ & 
\\ 
\qquad \qquad \qquad \textbf{endFor} &  \\ 
\qquad \qquad \qquad Return $Y_{1}^{k}$ &  \\ \hline
\end{tabular}

\begin{tabular}{|l|l|}
\hline
\textsf{Algorithm 3'} &  \\ \hline
Similar to Algorithm 3 with $t_{i}:=m_{\mathbf{X}}^{-1}(m_{i})$ substituted by &  \\ 
\qquad \qquad \qquad \qquad $t_{i}:=t_{i-1}+\frac{m_{f}\left(
t_{i-1}\right)-x_{i}}{\left( n-i\right)s_{f}^{2}\left( t_{i-1}\right) }$ &  \\ \hline
\end{tabular}
\end{center}

\begin{remark}
The r.v. $Y_{1}$ can be obtained through Metropolis-Hastings algorithm; see
also \cite{BarbeBroniatowski1999} which uses a truncated approximation.
\end{remark}

The following algorithm provides a simple acceptance/rejection simulation
tool for $Y_{i+1}$ with density $h_{i}(\left. y_{i+1}\right\vert y_{1}^{i});$
it does not require any estimation of the normalizing factor.
Metropolis-Hastings may also be useful in complex cases.

Denote $\mathfrak{N}$ the c.d.f. of a normal variate with parameter $\left(
\mu,\sigma^{2}\right) $ ,and $\mathfrak{N}^{-1}$ its inverse.

\begin{center}
\begin{tabular}{|l|l|}
\hline
\textsf{Algorithm 4 : Simulates }$Y$\textsf{\ with density proportional to }$%
p(x)\mathfrak{n}\left( \mu,\sigma^{2},x\right)$ &  \\ \hline
\textbf{INPUT} &  \\ 
\qquad\qquad\qquad\ density $p$ &  \\ 
\textbf{OUTPUT} &  \\ 
\qquad\qquad\qquad$Y$ &  \\ 
\textbf{INITIALIZATION} &  \\ 
\qquad\qquad\qquad Select a density $f$ on $\left[ 0,1\right] $ and a
positive constant $K$ such that &  \\ 
$\qquad\qquad\qquad p\left( \mathfrak{N}^{-1}(x)\right) \leq Kf(x)$ for all $%
x$ in $\left[ 0,1\right]$ &  \\ 
\textbf{PROCEDURE} &  \\ 
\qquad\qquad\qquad\textbf{Do} &  \\ 
\qquad\qquad\qquad\qquad Simulate $X$ with density $f$ &  \\ 
\qquad\qquad\qquad\qquad Simulate $U$ uniform on $\left[ 0,1\right]$
independent of $X$ &  \\ 
\qquad\qquad\qquad\qquad$Z:=KUf(X)$ &  \\ 
\qquad\qquad\qquad\textbf{While} $Z<$ $p\left( \mathfrak{N}^{-1}(X)\right)$
&  \\ 
\qquad\qquad\qquad\textbf{endDo} &  \\ 
\qquad\qquad\qquad Return $Y:=\mathfrak{N}^{-1}(X)$ &  \\ \hline
\end{tabular}
\end{center}

Tables 9,10,11 and 12 present a number of simulations of random walks
conditioned on their sum with $n=1000$ when $f(x)=x.$ In the gaussian case,
when the aproximating scheme is known to be optimal up to $k=n-1$, the
simulation is performed with $k=999$ and two cases are considered: the
moderate deviation case is when $P(\mathbf{S}_{1}^{n}>na)=10^{-2}$ (Table 9)
and the large deviation pertains to $P(\mathbf{S}_{1}^{n}>na)=10^{-8}$
(Table 10). The centered exponential case with $n=1000$ and $k=900$ is
presented in Tables 11 and 12, under the same events. In order to check the
accuracy of the approximation, Tables 13,14 (normal case, n=1000, k=999) and
Tables 15,16 (centered exponential case, n=1000, k=900) present the
histograms of the simulated $\mathbf{X}_{i}^{\prime}s$ together with the
tilted densities at point $a$ which are known to be the limit density of $%
\mathbf{X}_{1}$ conditioned on $\mathcal{E}_{n}$ in the large deviation
case, and to be equivalent to the same density in the moderate deviation
case, as can be deduced from \cite{ERmakov2006}.$\ $The tilted density in
the gaussian case is the normal with mean $a$ and variance $1$; in the
centered exponential case the tilted density is an exponential density on $%
\left( -1,\infty \right) $ with parameter $1/(1+a).$

\newpage

\begin{figure}[h!t]
\begin{minipage}[b]{0.50\linewidth}
      \centering \includegraphics*[scale=0.22]{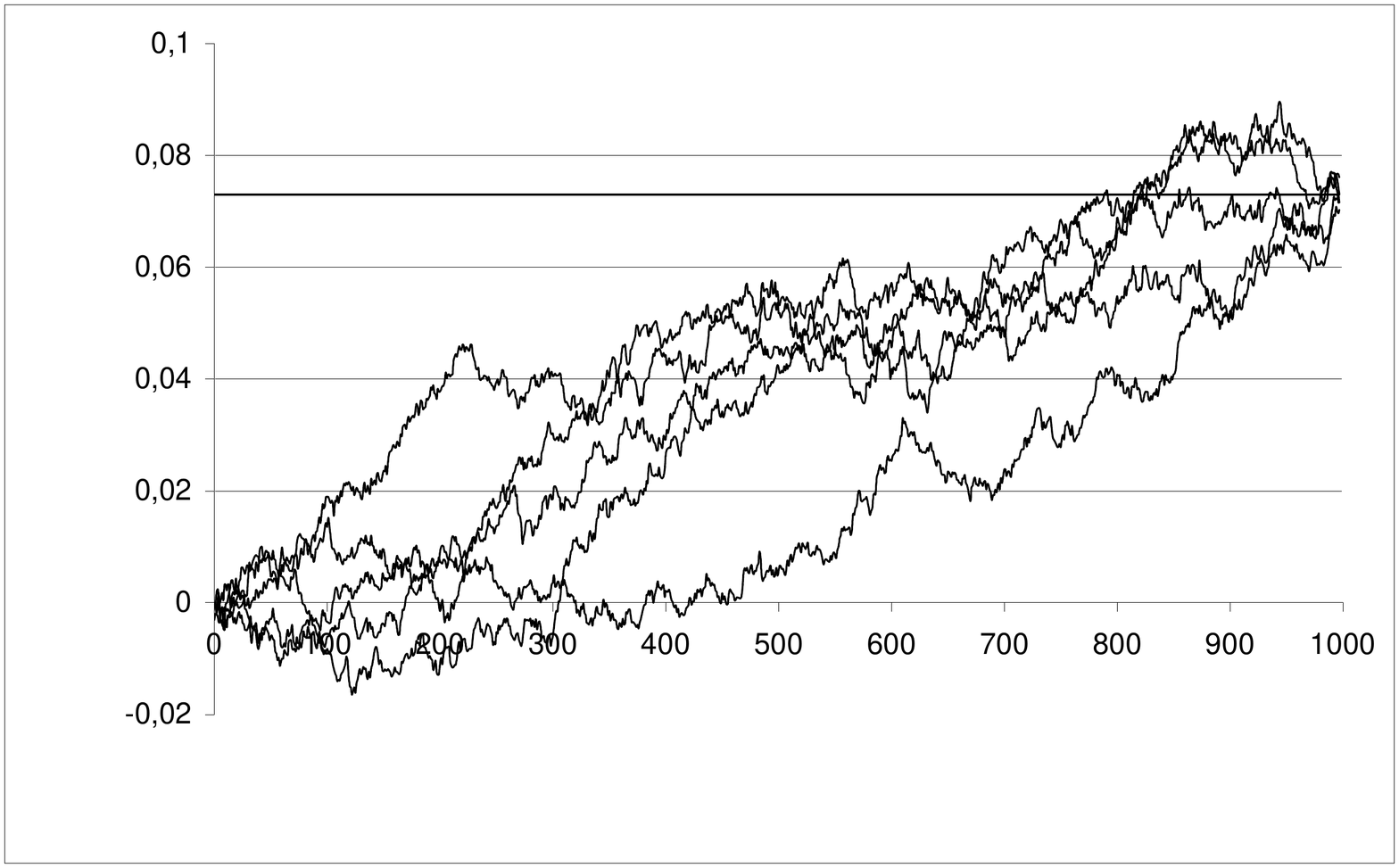}
      \caption{Trajectories in the normal\\ case for $P(\mathbf{S}_{1}^{n}>na)=10^{-2}$}
   \end{minipage}\hfill 
\begin{minipage}[b]{0.50\linewidth}   
      \centering \includegraphics*[scale=0.22]{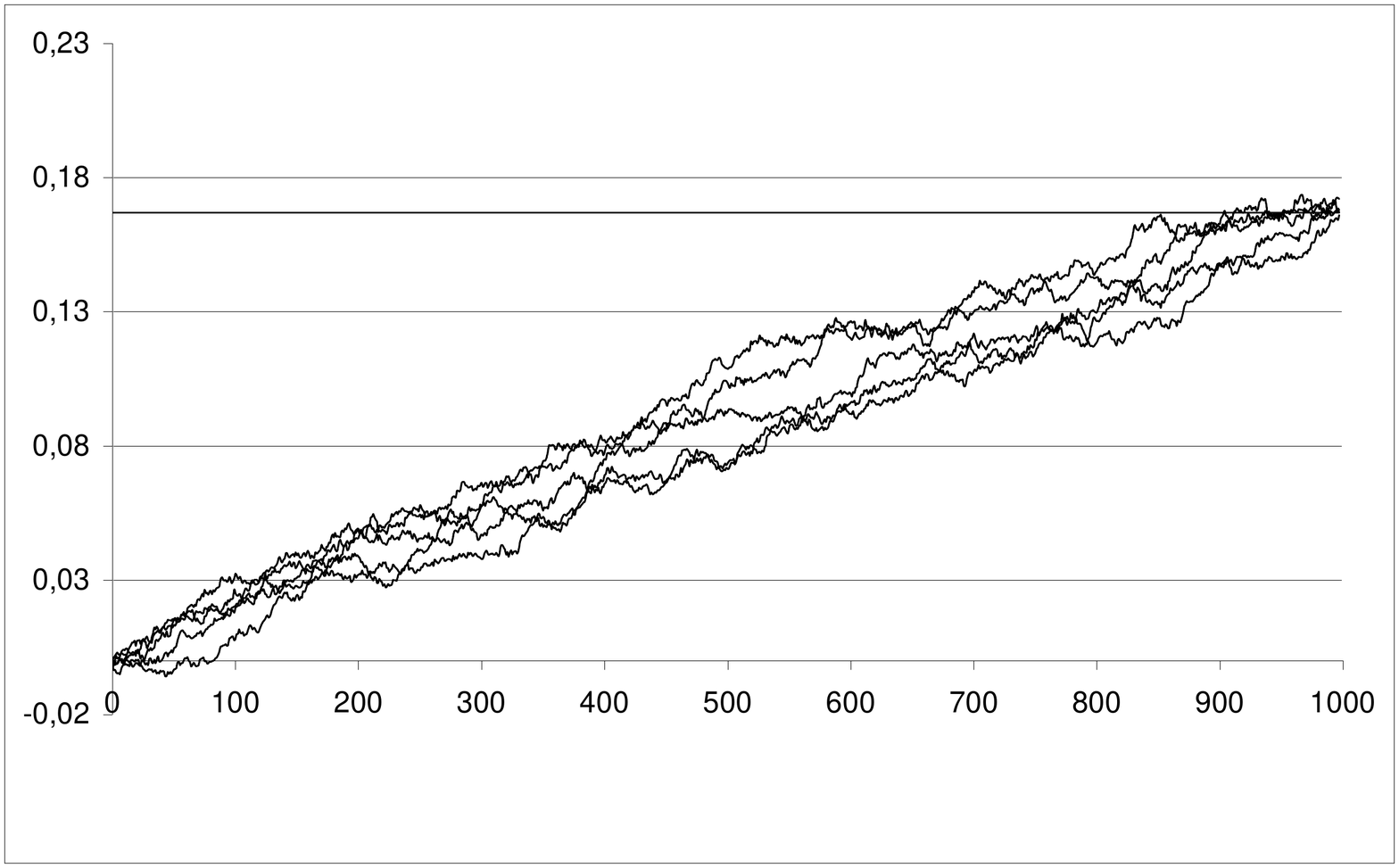}
      \caption{Trajectories in the normal\\ case for $P(\mathbf{S}_{1}^{n}>na)=10^{-8}$}
   \end{minipage}
\end{figure}

\begin{figure}[h!t]
\begin{minipage}[b]{0.50\linewidth}
      \centering \includegraphics*[scale=0.22]{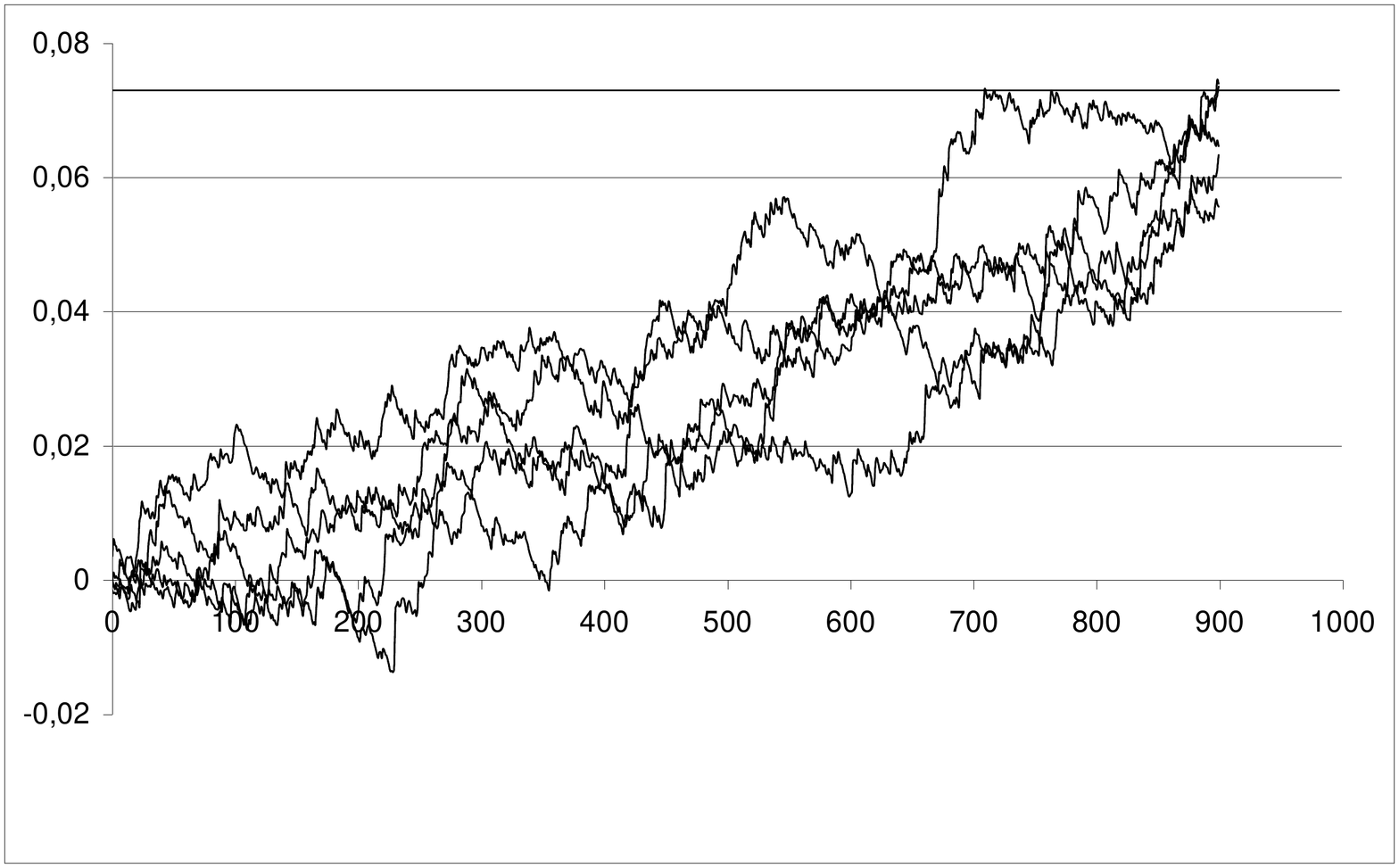}
      \caption{Trajectories in the expo-\\nential case for $P(\mathbf{S}_{1}^{n}>na)=10^{-2}$}
   \end{minipage}\hfill 
\begin{minipage}[b]{0.50\linewidth}   
      \centering \includegraphics*[scale=0.22]{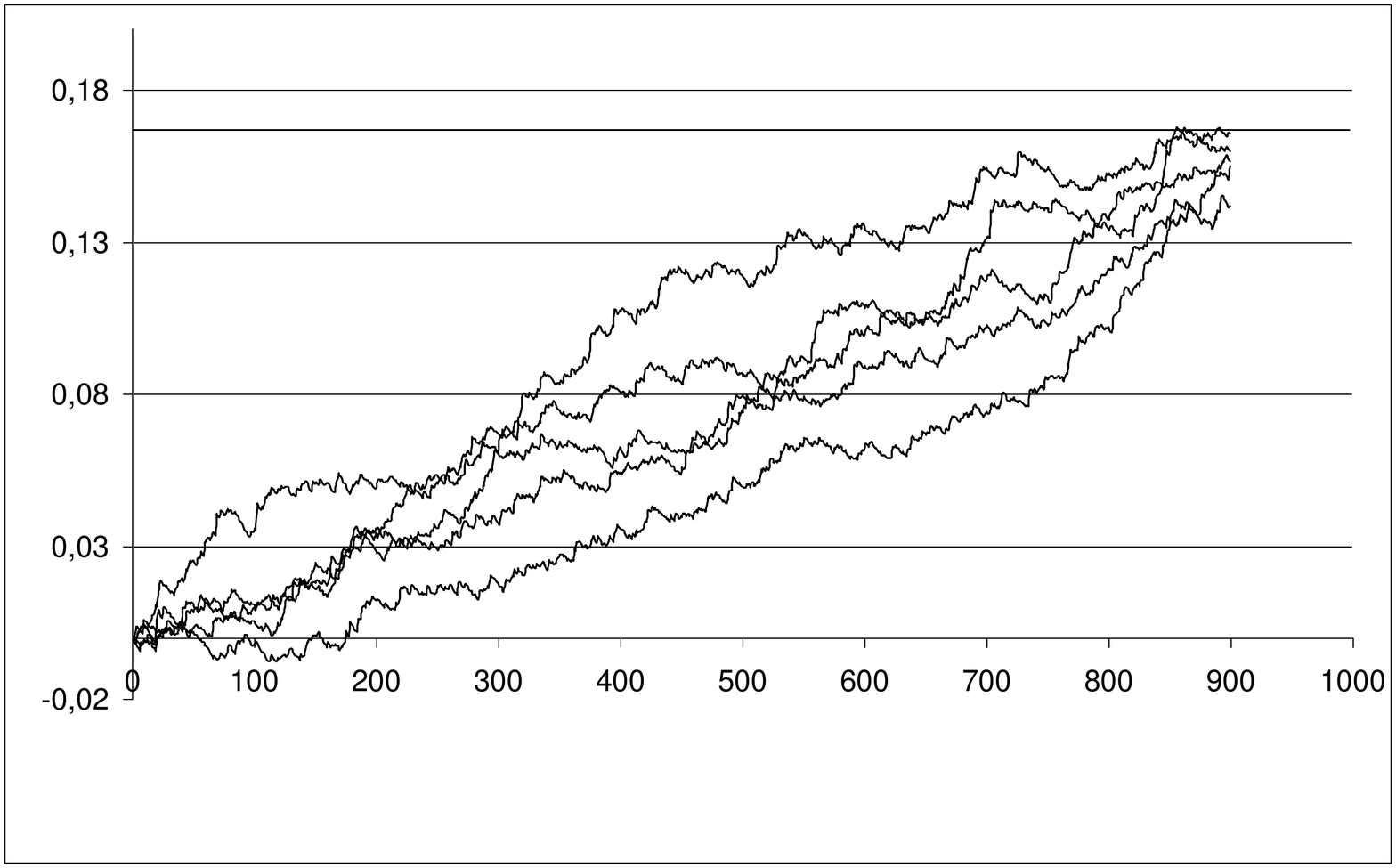}
      \caption{Trajectories in the expo-\\nential case for $P(\mathbf{S}_{1}^{n}>na)=10^{-8}$}
   \end{minipage}
\end{figure}

\newpage

\begin{figure}[h!t]
\begin{minipage}[b]{0.50\linewidth}
      \centering \includegraphics*[viewport=30 30 500 457,scale=0.3]{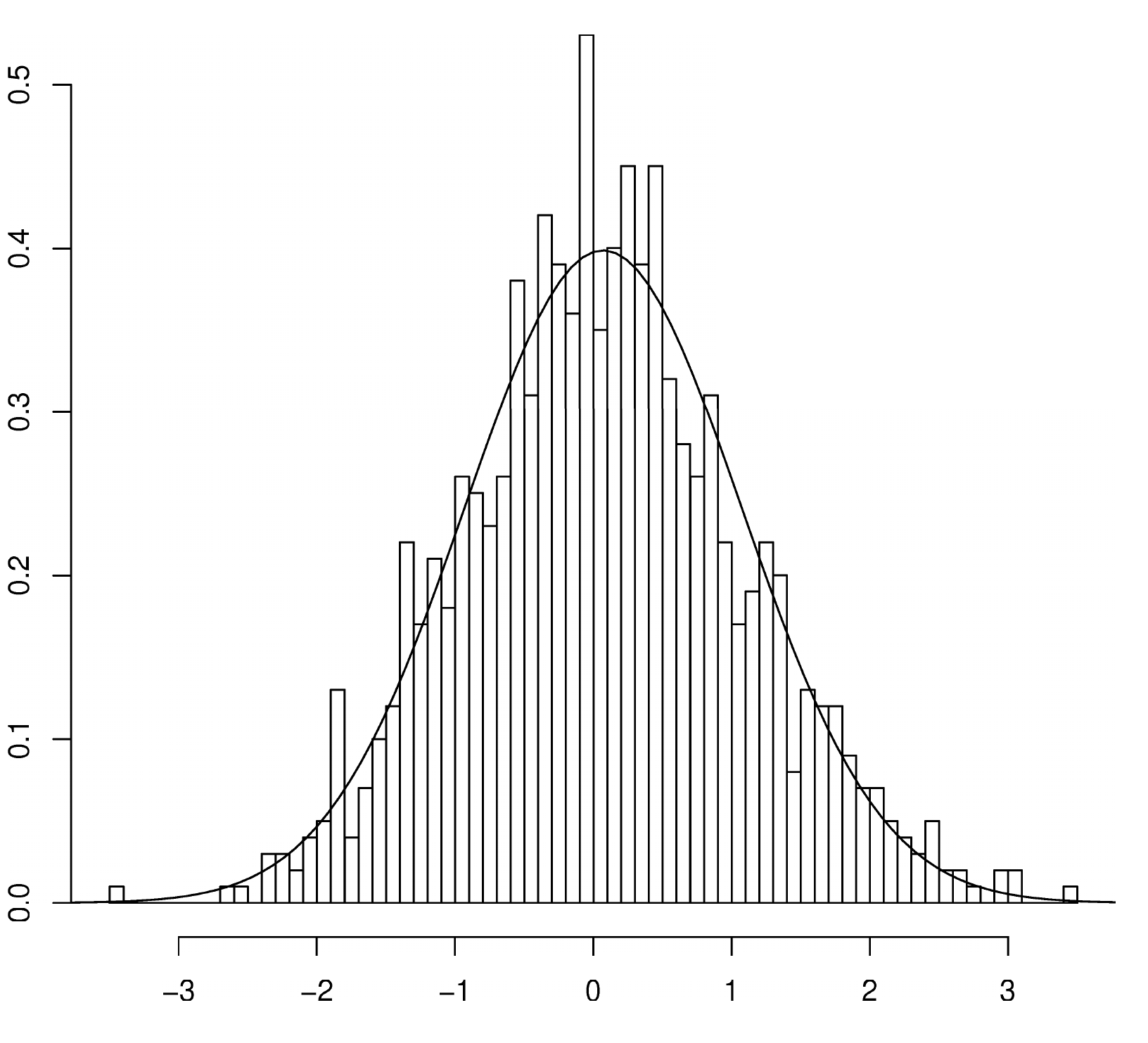}
      \caption{Distribution of $\textbf{X}_{i}'s$ in the\\ normal case for $P(\mathbf{S}_{1}^{n}>na)=10^{-2}$}
   \end{minipage}\hfill 
\begin{minipage}[b]{0.50\linewidth}   
      \centering \includegraphics*[viewport=65 210 510 620,scale=0.3]{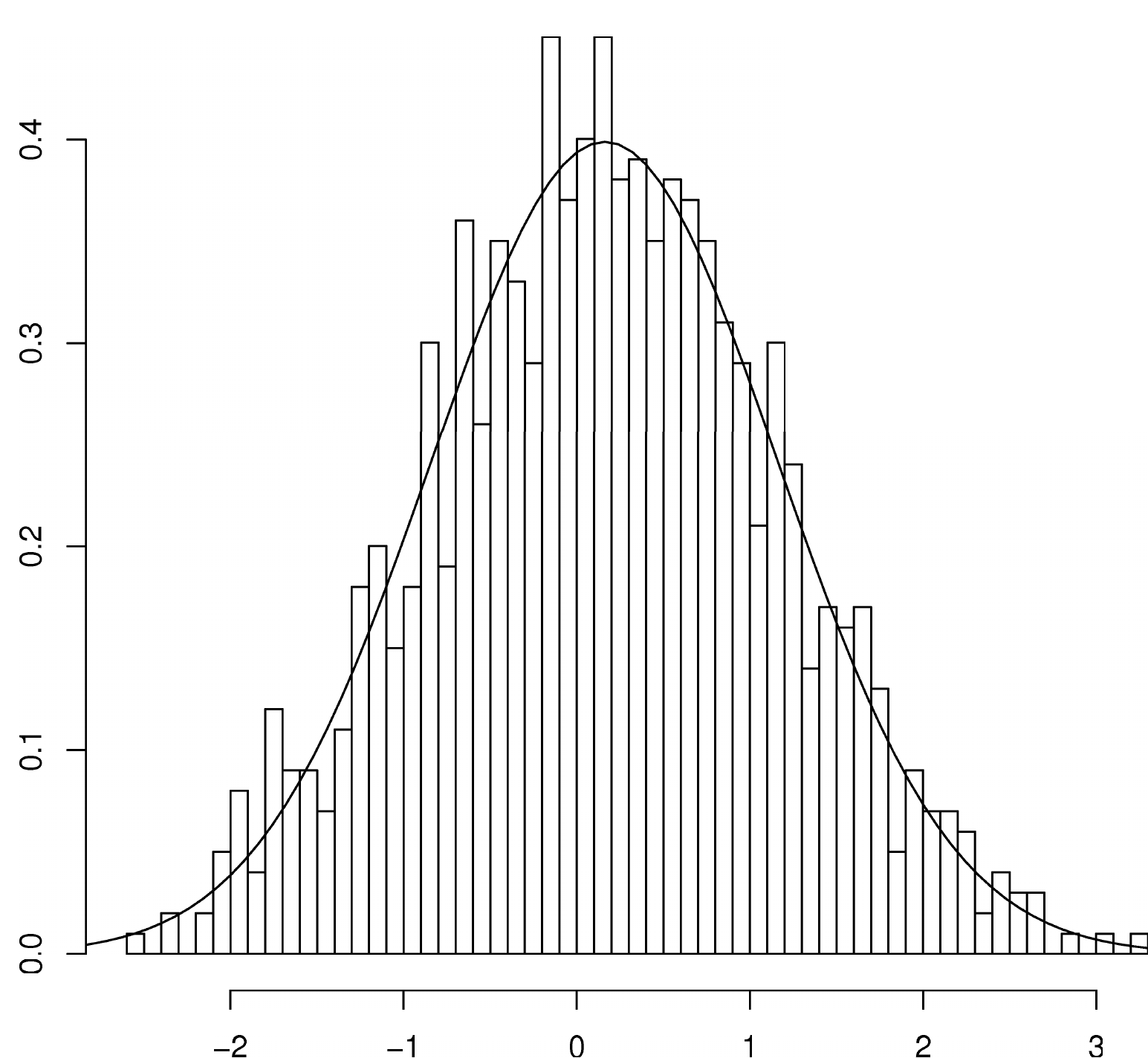}
      \caption{Distribution of $\textbf{X}_{i}'s$ in the\\ normal case for $P(\mathbf{S}_{1}^{n}>na)=10^{-8}$}
   \end{minipage}
\end{figure}

\begin{figure}[h!t]
\begin{minipage}[b]{0.50\linewidth}
      \centering \includegraphics*[viewport=65 210 510 620,scale=0.3]{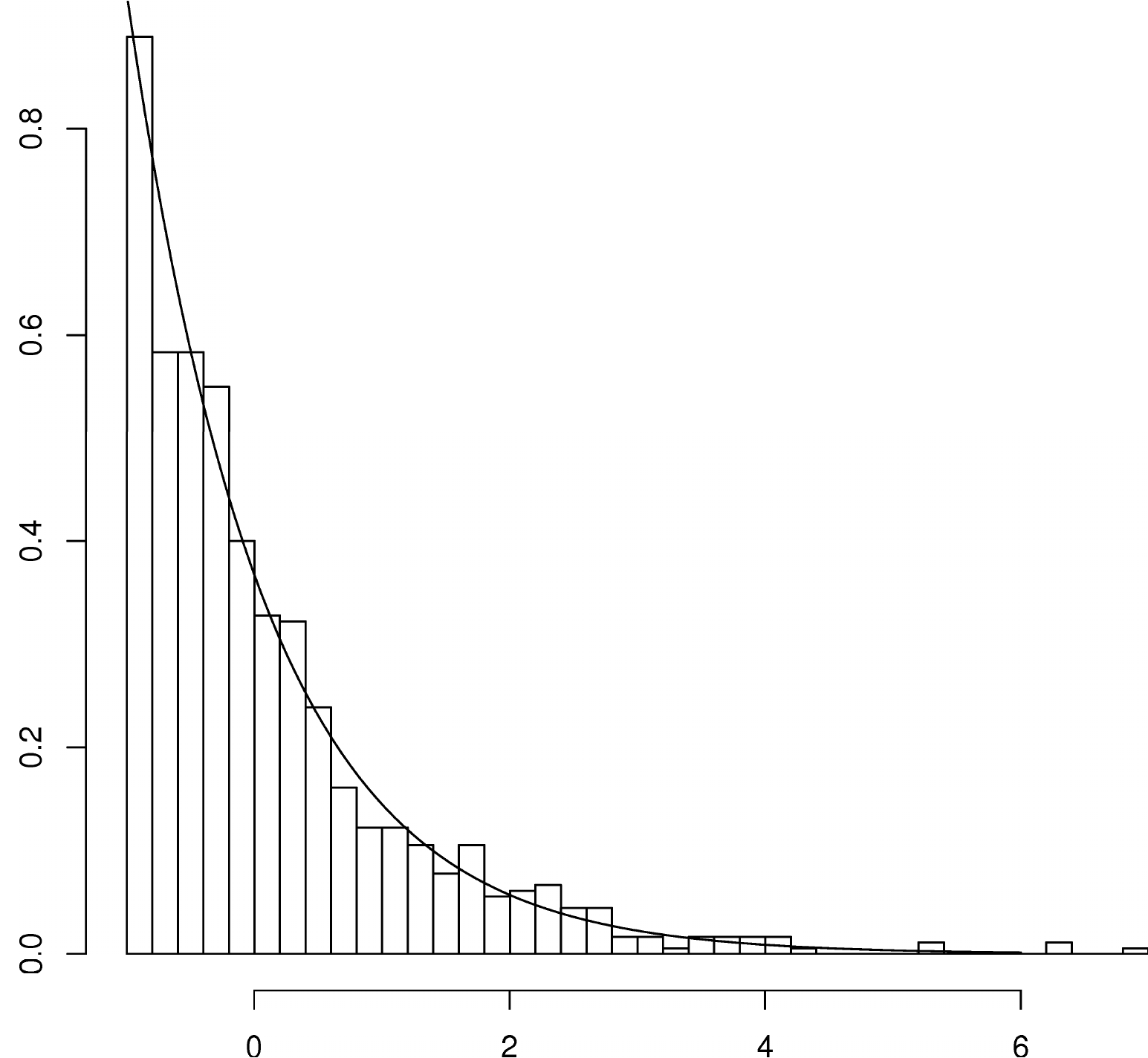}
      \caption{Distribution of $\textbf{X}_{i}'s$ in the\\ exponential case for\\ $P(\mathbf{S}_{1}^{n}>na)=10^{-2}$}
   \end{minipage}\hfill 
\begin{minipage}[b]{0.50\linewidth}   
      \centering \includegraphics*[viewport=65 210 510 620,scale=0.3]{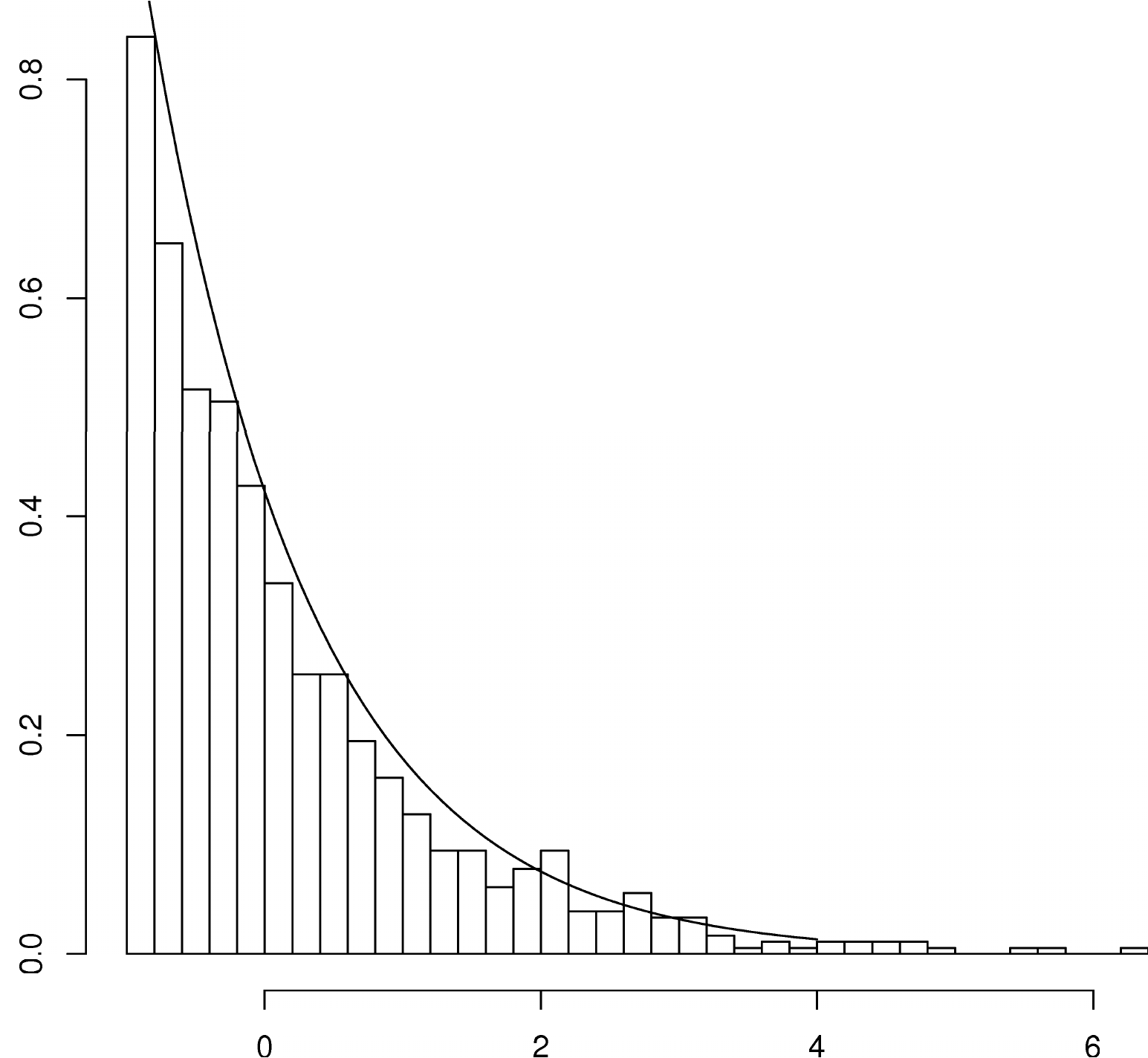}
      \caption{Distribution of $\textbf{X}_{i}'s$ in the\\ exponential case for\\ $P(\mathbf{S}_{1}^{n}>na)=10^{-8}$}
   \end{minipage}
\end{figure}

\newpage

Consider now the case when $f(x)=x^{2}.$ Table 1è presents the case when $%
\mathbf{X}$ is $N(0,1)$, $n=1000,k=800$, $P\left( \mathbf{U}_{1}^{n}=n\left(
a\sqrt{2}+1\right)\right) \simeq 10^{-2}.$ We present the histograms of the $%
X_{i}^{\prime }s$ together with the graph of the corresponding titlted
density; when $\mathbf{X}$ is $N(0,1)$ then $\mathbf{X}^{2}$ is $\chi ^{2}.$
It is well known that when $a$ is fixed larger than $1$ then the limit
distribution of $\mathbf{X}_{1}$ conditioned on $\left( \mathbf{U}%
_{1}^{n}=n\left( a\sqrt{2}+1\right) \right) $ tends to $N\left( 0,1+a\sqrt{2}%
\right) $ which is the Kullback-Leibler projection of $N(0,1)$ on the set of
all probability measures $Q$ on $\mathbb{R}$ with $\int x^{2}dQ(x)=a\sqrt{2}%
+1.$ Now this distribution is precisely $h_{0}(\left. y_{1}\right\vert
y_{0}) $ defined hereabove. Also consider (\ref{gif}); expansion using the
definitions (\ref{a pour f(x)}) and (\ref{b pour f(x)}) prove that as $%
n\rightarrow \infty $ the dominating term in $h_{i}(\left.
y_{i+1}\right\vert y_{1}^{i})$ is precisely $N\left( 0,1+a\sqrt{2}\right) ,$
and the terms including $y_{i+1}^{4}$ in the exponential stemming from $%
\mathfrak{n}\left( \alpha \beta +\left(\sigma a+\mu\right),\alpha ,f(y_{i+1}) \right) $ are of order $O\left( 1/\left( n-i\right) \right) $; the
terms depending on $y_{1}^{i}$ are of smaller order$.$ The fit which is
observed in Table 17 is in concordance with the above statement in the LDP
range (fixed $a$), and with the MDP approximation following Ermakov; see 
\cite{ERmakov2006} .

\begin{figure}[h!t]
\begin{minipage}[b]{0.50\linewidth}
      \centering \includegraphics*[viewport=65 210 510 620,scale=0.3]{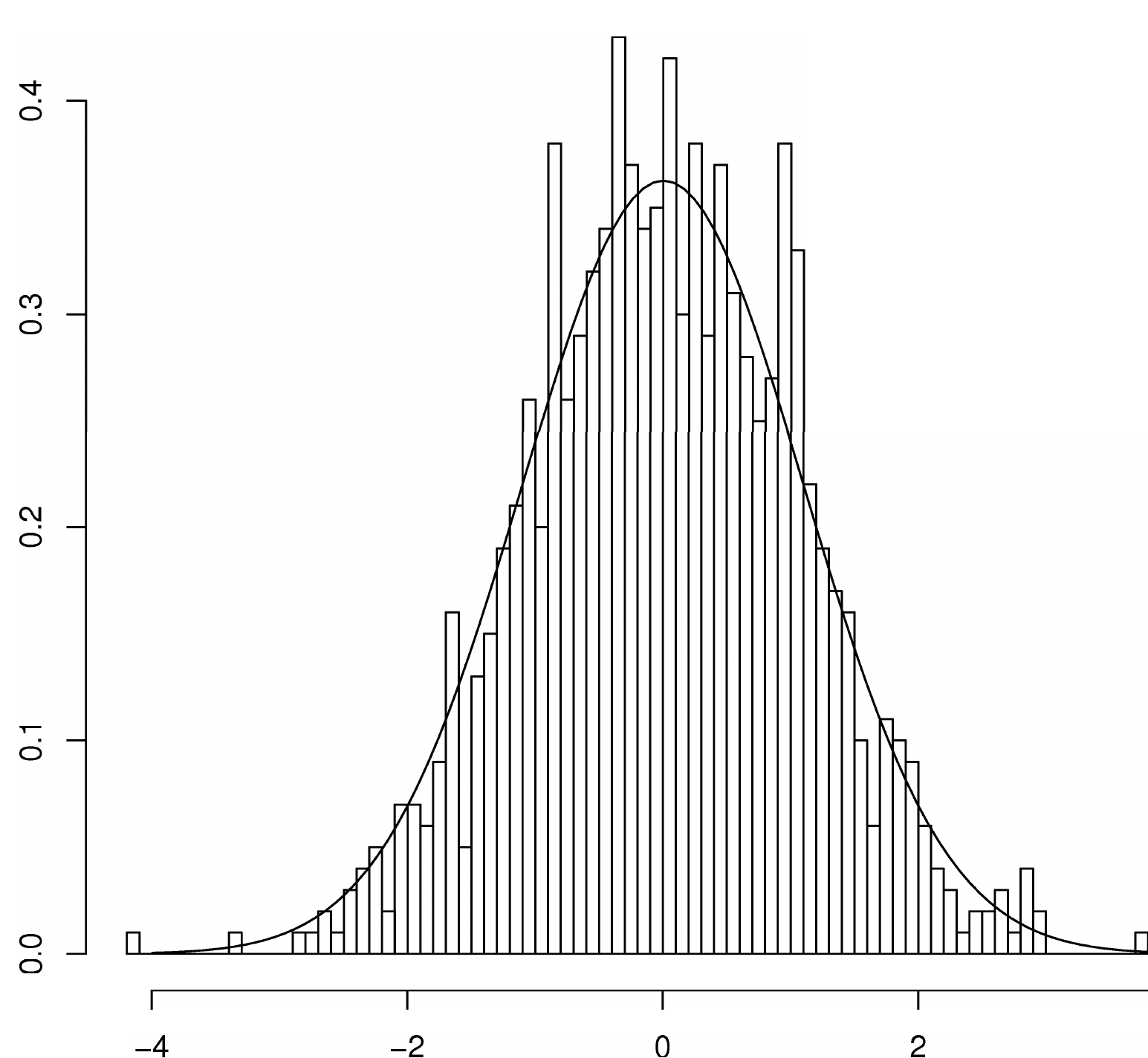}
      \caption{Distribution of $\textbf{X}_{i}'s$ in the\\ normal case for $P(\mathbf{S}_{1}^{n}>na)=10^{-2}$\\ and for $f(x)=x^{2}$}
   \end{minipage}\hfill
\end{figure}

\begin{remark}
The statistics $\mathbf{U}_{1}^{n}$ may be substituted by any regular
M-estimator when $a=a_{n}$ defines a moderate deviation event, namely when
(A) holds together with $a_{n}\rightarrow 0.$ In this case it is well known
that the distribution of any regular $M$-estimator is similar to that of the
mean of its influence function evaluated on the sample points; this allows
to simulate samples under a given model conditionally on an observed value
of a statistics, when observed in a rare area under the model.
\end{remark}

\section{Conclusion}

We have obtained an extended version of Gibbs conditional principle in the
simple case of real valued independent r.v's conditioned on the value of
their mean or on an average of their images through some real valued
function. The approximation of the density of long runs is shown to be quite
accurate and can be controlled through an explicit rule. Algorithms for the
simulation of these runs are presented, together with numerical examples.
Applications to Importance Sampling procedures for rare event simulation is
a first application of this scheme; it mainly requires to consider
conditioning events of the form $\left( \mathbf{S}_{1}^{n}>n\left( \sigma
a_{n}+\mu\right) \right) $ instead of $\left( \mathbf{S}_{1}^{n}=n\left(
\sigma a_{n}+\mu\right) \right) $; first numerical results obtained in \cite%
{BroniatowskiRitov2009} show a net gain in the variance for IS estimators
using an extension of the present scheme. Extension to the multivariate
setting is obtainable, requiring slight modifications. The case when the
conditioning event is in the CLT\ zone deserves attention due to its
interest to statistics. Simulation of $\mathbf{X}_{1}^{n}$ under a given $%
p_{\theta _{0}}$ in a model $\left( p_{\theta },\theta \in \Theta \right) $
may lead to conditional test when $a$ is substituted by the observed value
of a given statistics. The present case when conditioning under a moderate
deviation event is of interest for accurate assesment when the observed
statistics has small $p-$ value under a given hypothesis. Some other
potential application to statistics is related to test procedures in
presence of nuisance parameter, considering conditional tests under a
sufficient statistics for the nuisance.

\appendix{}

\section{Three Lemmas pertaining to the partial sum under its final value}

We state three lemmas which describe some functions of the random vector $%
\mathbf{X}_{1}^{n}$ conditioned on $\mathcal{E}_{n}$. The r.v. $\mathbf{X}$
is assumed to have expectation $0$ and variance $1.$

\begin{lemma}
\label{LemmaMomentsunderE_n}It holds $E_{\mathfrak{P}_{n}}\left( \mathbf{X}%
_{1}\right) =a,E_{\mathfrak{P}_{n}}\left( \mathbf{X}_{1}\mathbf{X}%
_{2}\right) =a{}^{2}+0\left( \frac{1}{n}\right) .$ $E_{\mathfrak{P}%
_{n}}\left( \mathbf{X}_{1}^{2}\right) =s^{2}(t)+a{}^{2}+0\left( \frac{1}{n}%
\right) $ where $m(t)=a.$
\end{lemma}

\begin{proof}
Using%
\begin{equation*}
\mathfrak{P}_{n}(\mathbf{X}_{1}=x)=\frac{p_{\mathbf{S}_{2}^{n}}\left(
na-x\right) p_{\mathbf{X}_{1}}(x)}{p_{\mathbf{S}_{1}^{n}}\left( na\right) }=%
\frac{\pi_{\mathbf{S}_{2}^{n}}^{a}\left( na-x\right) \pi_{\mathbf{X}%
_{1}}^{a}(x)}{\pi_{\mathbf{S}_{1}^{n}}^{a}\left( na\right) }
\end{equation*}
normalizing both $\pi_{\mathbf{S}_{2}^{n}}^{a}\left( na-x\right) $ and $\pi_{%
\mathbf{S}_{1}^{n}}^{a}\left( na\right) $ and making use of a first order
Edgeworth expansion in those expressions yields the asymptotic expressions
for $E_{\mathfrak{P}_{n}}\left( \mathbf{X}_{1}^{2}\right)
=s^{2}(t)+a{}^{2}+0\left( \frac{1}{n}\right) $ here above. $\ $A similar
development for the joint density $\mathfrak{P}_{n}(\mathbf{X}_{1}=x,\mathbf{%
X}_{2}=y)$, using the same tilted distribution $\pi^{a}$ it readily follows
that the last result holds. We used the fact that $a_{n}$ is a bounded
sequence.
\end{proof}

The following result states the behavior of the moments of $\pi ^{m_{i}}.$

\begin{lemma}
\label{LemmaMaxm_in} Assume (A) and (E1). Then $\max_{1\leq i\leq
k}\left\vert m_{i}\right\vert =a+o_{\mathfrak{P}_{n}}\left(
\epsilon_{n}\right) .$ Also $\max_{1\leq i\leq k}s_{i}^{2}$, $\max_{1\leq
i\leq k}\mu_{3}^{i}$ and $\max_{1\leq i\leq k}\mu_{4}^{i}$ tend in $%
\mathfrak{P}_{n}$ probability to the variance, skewness and kurtosis of $p$
when $a=a_{n}\rightarrow0$ and remain bounded when $a$ is fixed positive.
\end{lemma}

\begin{proof}
Define 
\begin{align*}
V_{i+1} & :=m(t_{i})-a \\
& =\frac{\Sigma_{i+1}^{n}}{n-i}-a.
\end{align*}
We state that 
\begin{equation}
{\max_{0\leq{i\leq{k-1}}}|V_{i+1}|=}o_{\mathfrak{P}_{n}}\left( \epsilon
_{n}\right) ,  \label{maxV_i+1}
\end{equation}
namely for all positive $\delta$ 
\begin{equation*}
\lim_{n\rightarrow\infty}\mathfrak{P}_{n}\left( \max_{0\leq{i\leq{k-1}}%
}|V_{i+1}|>\delta\epsilon_{n}\right) =0
\end{equation*}
which we prove following Kolmogorov maximal inequality. Define 
\begin{equation*}
A_{i}:=\left( \left( |V_{i+1}|\geq\delta\epsilon_{n}\right) \text{ and }%
\left( |V_{j}|<\delta\epsilon_{n}\text{ for all }j<i+1\right) \right) .
\end{equation*}
from which 
\begin{equation*}
\left( \max_{0\leq{i}\leq{k-1}}|V_{i+1}|>\delta\epsilon_{n}\right)
=\bigcup_{i=0}^{k-1}A_{i}.
\end{equation*}
It holds 
\begin{align*}
E_{\mathfrak{P}_{n}}V_{k}^{2} & ={\int_{\cup A_{i}}V_{k}^{2}d\mathfrak{P}%
_{n}+\int_{\left( \cup A_{i}\right) ^{c}}V_{k}^{2}d}\mathfrak{P}_{n} \\
& \geq{\int_{\cup A_{i}}}\left( {V_{i}^{2}+2}\left( V_{k}-V_{i}\right)
V_{i}\right) {d\mathfrak{P}_{n}+\int_{\left( \cup A_{i}\right) ^{c}}}\left( {%
V_{i}^{2}+2}\left( V_{k}-V_{i}\right) V_{i}\right) {d}\mathfrak{P}_{n} \\
& \geq{\int_{\cup A_{i}}V_{i}^{2}d}\mathfrak{P}_{n} \\
& \geq\delta^{2}\epsilon_{n}^{2}{\sum_{j=0}^{k-1}\mathfrak{P}_{n}(A_{j})} \\
& =\delta^{2}\epsilon_{n}^{2}\mathfrak{P}_{n}\left( {\max_{0\leq{i}\leq {k-1}%
}|V_{i+1}|>}\delta\epsilon_{n}\right) .
\end{align*}

The third line above follows from $EV_{i}\left( V_{k}-V_{i}\right) =0$ which
is proved hereunder.\ Hence 
\begin{equation*}
\mathfrak{P}_{n}\left( {\max_{0\leq{i}\leq{k-1}}|V_{i+1}|>}\delta\epsilon
_{n}\right) \leq\frac{Var_{\mathfrak{P}_{n}}(V_{k})}{\delta^{2}\epsilon
_{n}^{2}}=\frac{1}{\delta^{2}\epsilon_{n}^{2}\left( n-k\right) }(1+o(1))
\end{equation*}
where we used Lemma \ref{LemmaMomentsunderE_n}; therefore (\ref{maxV_i+1})
holds under (E1). By direct calculation, we can show that $E_{\mathfrak{P}_{n}}\left(V_{i}(V_{k}-V_{i})\right) ={0}$, 
which achieves the proof.
\end{proof}

We also need the order of magnitude of $\max\left( \mathbf{X}_{1},...,%
\mathbf{X}_{k}\right) $ under $\mathfrak{P}_{n}$ which is stated in the
following result.

\begin{lemma}
\label{Lemma max X_i under conditioning} For all $k$ between $1$ and $%
n,\max\left( \mathbf{X}_{1},...,\mathbf{X}_{k}\right) =O_{\mathfrak{P}%
_{n}}\left( \log k\right) .$
\end{lemma}

\begin{proof}
For all $t$ it holds%
\begin{align*}
\mathfrak{P}_{n}\left( \max\left( \mathbf{X}_{1},...,\mathbf{X}_{k}\right)
>t\right) & \leq k\mathfrak{P}_{n}\left( \mathbf{X}_{n}>t\right) \\
& =k\int_{t}^{\infty}\pi^{a}\left( \mathbf{X}_{n}=u\right) \frac{\pi ^{a}(%
\mathbf{S}_{1}^{n-1}=na_{n}-u)}{\pi^{a}\left( \mathbf{S}_{1}^{n}=na_{n}%
\right) }du.
\end{align*}
Let $\tau$ be such that $m(\tau)=a.$ Center and normalize both $\mathbf{S}%
_{1}^{n}$ and $\mathbf{S}_{1}^{n-1}$with respect to the density $\pi^{a}$ in
the last line above, denoting $\overline{\pi_{n}^{a}}$ the density of $%
\overline{\mathbf{S}_{1}^{n}}:=\left( \mathbf{S}_{1}^{n}-na_{n}\right)
/s^{(a)}\sqrt{n}$ when $\mathbf{X}$ has density $\pi^{a}$ with mean $a$ and
variance $\left( s^{(a)}\right) ^{2},$ we get 
\begin{align*}
\mathfrak{P}_{n}\left( \max\left( \mathbf{X}_{1},...,\mathbf{X}_{k}\right)
>t\right) & \leq k\frac{\sqrt{n}}{\sqrt{n-1}}\int_{t}^{\infty}\pi ^{a}\left( 
\mathbf{X}_{n}=u\right) \\
& \frac{\overline{\pi_{n-1}^{a}}\left( \overline{\mathbf{S}_{1}^{n-1}}%
=\left( na-u-(n-1)a\right) /\left( s^{(a)}\sqrt{n-1}\right) \right) }{%
\overline{\pi_{n}^{a}}\left( \overline{\mathbf{S}_{1}^{n}}=0\right) }du.
\end{align*}
Under the sequence of densities $\pi^{a}$ the triangular array $\left( 
\mathbf{X}_{1},...,\mathbf{X}_{n}\right) $ obeys a first order Edgeworth
expansion 
\begin{align*}
\mathfrak{P}_{n}\left( \max\left( \mathbf{X}_{1},...,\mathbf{X}_{k}\right)
>t\right) & \leq k\frac{\sqrt{n}}{\sqrt{n-1}}\int_{t}^{\infty}\pi ^{a}\left( 
\mathbf{X}_{n}=u\right) \\
& \frac{\mathfrak{n}\left( \left( a-u\right) /s^{(a)}\sqrt{n-1}\right) 
\mathbf{P}\left( u,i,n\right) +o(1)}{\mathfrak{n}\left( 0\right) +o(1)}du \\
& \leq kCst\int_{t}^{\infty}\pi^{a}\left( \mathbf{X}_{n}=u\right) du.
\end{align*}
for some constant $Cst$ independent of $n$ and $\tau$ and 
\begin{equation*}
\mathbf{P}\left( u,i,n\right) :=1+P_{3}\left( \left( a-u\right) /s^{(a)}%
\sqrt{n-1}\right)
\end{equation*}
where $P_{3}(x)=\frac{\mu_{3}^{(a)}}{6\left( s^{(a)}\right) ^{3}}\left(
x^{3}-3x\right) $ is the third Hermite polynomial; $\mu_{3}^{(a)}$ is the
third centered moment of $\pi^{a}.$ We used uniformity upon $u$ in the
remaining term of the Edgeworth expansions. Making use of Chernoff
Inequality to bound $\Pi^{a}\left( \mathbf{X}_{n}>t\right) ,$%
\begin{equation*}
\mathfrak{P}_{n}\left( \max\left( \mathbf{X}_{1},...,\mathbf{X}_{k}\right)
>t\right) \leq kCst\frac{\Phi(t+\lambda)}{\Phi(t)}e^{-\lambda t}
\end{equation*}
for any $\lambda$ such that $\phi(t+\lambda)$ is finite. For $t$ such
that 
\begin{equation*}
t/\log k\rightarrow\infty
\end{equation*}
it holds 
\begin{equation*}
\mathfrak{P}_{n}\left( \max\left( \mathbf{X}_{1},...,\mathbf{X}_{k}\right)
<t\right) \rightarrow1,
\end{equation*}
which proves the lemma.\bigskip
\end{proof}

\section{Proof of the approximations resulting from Edgeworth expansions in
Theorem 1}

We complete the calculation leading to (\ref{Adem}) and (\ref{Bdem}).

Set $Z_{i+1}:=\left( m_{i}-Y_{i+1}\right) /s_{i}\sqrt{n-i-1}.$

It then holds%
\begin{align}
\overline{\pi_{n-i-1}}\left( Z_{i+1}\right) & =\mathfrak{n}(Z_{i+1})\left[ 
\begin{array}{c}
1+\frac{1}{\sqrt{n-i-1}}P_{3}(Z_{i+1})+\frac{1}{n-i-1}P_{4}(Z_{i+1}) \\ 
+\frac{1}{\left( n-i-1\right) ^{3/2}}P_{5}(Z_{i+1})%
\end{array}
\right]  \label{Hermite} \\
& +O_{\mathfrak{P}_{n}}\left( \frac{1}{\left( n-i-1\right) ^{3/2}}\right) . 
\notag
\end{align}

We perform an expansion in $\mathfrak{n}(Z_{i+1})$ up to the order $3,$ with
a first order term $\mathfrak{n}\left( -Y_{i+1}/\left( s_{i}\sqrt {n-i-1}%
\right) \right) ,$ namely%
\begin{align}
\mathfrak{n}(Z_{i+1}) & =\mathfrak{n}\left( -Y_{i+1}/\left( s_{i}\sqrt{n-i-1}%
\right) \right)  \label{approx gauss} \\
& \left( 
\begin{array}{c}
1+\frac{Y_{i+1}m_{i}}{s_{i}^{2}\left( n-i-1\right) }+\frac{m_{i}^{2}}{%
2s_{i}^{2}\left( n-i-1\right) }\left( \frac{Y_{i+1}^{2}}{s_{i}^{2}\left(
n-i-1\right) }-1\right) \\ 
+\frac{m_{i}^{3}}{6s_{i}^{3}\left( n-i-1\right) ^{3/2}}\frac{\mathfrak{n}%
^{(3)}\left( \frac{Y^{\ast}}{\left( s_{i}\sqrt{n-i-1}\right) }\right) }{%
\mathfrak{n}\left( -Y_{i+1}/\left( s_{i}\sqrt{n-i-1}\right) \right) }%
\end{array}
\right)  \notag
\end{align}
where $Y^{\ast}=\frac{1}{s_{i}\sqrt{n-i-1}}(-Y_{i+1}+\theta m_{i})$ with $%
\left\vert \theta\right\vert <1.$

Lemmas \ref{LemmaMaxm_in} and \ref{Lemma max X_i under conditioning} provide
the orders of magnitude of the random terms in the above displays when
sampling under $\mathfrak{P}_{n}.$

Use those lemmas to obtain 
\begin{equation}
\frac{Y_{i+1}m_{i}}{s_{i}^{2}\left( n-i-1\right) }=\frac{Y_{i+1}}{n-i-1}%
\left( a+o_{\mathfrak{P}_{n}}\left( \epsilon_{n}\right) \right)
\label{control 1}
\end{equation}
and%
\begin{equation*}
\frac{m_{i}^{2}}{s_{i}^{2}\left( n-i-1\right) }=\frac{1}{n-i-1}\left( a+o_{%
\mathfrak{P}_{n}}\left( \epsilon_{n}\right) \right) ^{2}.
\end{equation*}
Also when (A) holds then the dominant terms in the bracket in (\ref{approx
gauss}) are precisely those in the two displays just above. This yields 
\begin{equation}
\mathfrak{n}(Z_{i+1})=\mathfrak{n}\left( \frac{-Y_{i+1}}{s_{i}\sqrt{n-i-1}}%
\right) \left( 
\begin{array}{c}
1+\frac{aY_{i+1}}{s_{i}^{2}(n-i-1)}-\frac{a^{2}}{2s_{i}^{2}(n-i-1)} \\ 
+\frac{o_{\mathfrak{P}_{n}}(\epsilon_{n}\log n)}{n-i-1}%
\end{array}
\right) .  \notag
\end{equation}

We now need a precise evaluation of the terms in the Hermite polynomials in (%
\ref{Hermite}).\ This is achieved using Lemmas \ref{LemmaMaxm_in} and \ref%
{Lemma max X_i under conditioning} which provide uniformity upon $i$ between 
$1$ and $k=k_{n}$ in all terms depending on the sample path $Y_{1}^{k}.$ The
Hermite polynomials depend upon the moments of the underlying density $%
\pi^{m_{i}}.$ Since $\overline{\pi_{1}^{m_{i}}}$ has expectation $0$ and
variance $1$ the terms corresponding to $P_{1}$ \ and $P_{2}$ vanish. Up to
the order $4$ the polynomials write $P_{3}(x)=\frac{\mu_{3}^{(i)}}{6\left(
s_{i}\right) ^{3}}H_{3}(x)$, $P_{4}(x)=\frac{(\mu_{3}^{i})^{2}}{72\left(
s_{i}\right) ^{6}}H_{6}(x)+\frac{\mu_{4}^{(i,n)}-3\left( s_{i}\right) ^{4}}{%
24\left( s_{i}\right) ^{4}}H_{4}(x) $ with $H_{3}(x):=x^{3}-3x$, $H_{4}(x):=
x^{4}+6x^{2}-3$ and $H_{6}(x):=x^{6}-15x^{4}+45x^{2}-15.$

Using Lemma \ref{LemmaMaxm_in} it appears that the terms in $x^{j},$ $j\geq
3 $ in $P_{3}$ and $P_{4}$ will play no role in the asymptotic behavior in (%
\ref{Hermite}) with respect to the constant term in $P_{4}$ and the term in $%
x$ from $P_{3}\ .$ Indeed substituting $x$ by $Z_{i+1}$ and dividing by $%
n-i-1$, the term in $x^{2}$ in $P_{4}$ writes $O_{\mathfrak{P}_{n}}\left(
\log n\right) ^{2}/(n-i)^{2}$ where we used Lemma \ref{LemmaMaxm_in}. These
terms are of smaller order than the term $-3x$ in $P_{3}$ which writes $-%
\frac{\mu_{3}^{i}}{2s_{i}^{4}\left( n-i-1\right) }\left( a-Y_{i+1}\right) =%
\frac{1}{n-i-1}O_{\mathfrak{P}_{n}}\left( \log n\right) .$

It holds%
\begin{align*}
\frac{P_{3}(Z_{i+1})}{\sqrt{n-i-1}} & =-\frac{\mu_{3}^{i}}{2s_{i}^{4}\left(
n-i-1\right) }\left( m_{i}-Y_{i+1}\right) \\
& \text{ }+\frac{\mu_{3}^{i}\left( m_{i}-Y_{i+1}\right) ^{3}}{6\left(
s_{i}\right) ^{6}(n-i-1)^{2}}
\end{align*}
which yields 
\begin{equation}
\frac{P_{3}(Z_{i+1})}{\sqrt{n-i-1}}=-\frac{\mu_{3}^{i}}{2s_{i}^{4}\left(
n-i-1\right) }\left( a-Y_{i+1}\right) +\frac{1}{\left( n-i-1\right) ^{2}}O_{%
\mathfrak{P}_{n}}\left( \log n\right) ^{3}.  \label{P3}
\end{equation}
For the term of order $4$ it holds 
\begin{equation*}
\frac{P_{4}(Z_{i+1})}{n-i-1}=\frac{1}{n-i-1}\left( \frac{(\mu_{3}^{i})^{2}}{%
72s_{i}^{6}} H_{6}(Z_{i+1})+\frac{\mu_{4}^{i}-3s_{i}^{4}}{24s_{i}^{4}}%
H_{4}(Z_{i+1})\right)
\end{equation*}
which yields 
\begin{equation}
\frac{P_{4}(Z_{i+1})}{n-i-1}=-\frac{\mu_{4}^{i}-3s_{i}^{4}}{8s_{i}^{4}\left(
n-i-1\right)}-\frac{15(\mu_{3}^{i})^{2}}{72s_{i}^{6}(n-i-1)}+\frac{O_{\mathfrak{P}%
_{n}}\left( (\log n)^{2}\right)}{\left(n-i-1\right) ^{2}} .  \label{P4}
\end{equation}
The fifth term in the expansion plays no role in the asymptotics, under (A).

To sum up , under (A), and comparing the remainder terms in (\ref{P3}) and (%
\ref{P4}), we get%
\begin{equation*}
\overline{\pi_{n-i-1}}\left( Z_{i+1}\right) =\mathfrak{n}\left(
-Y_{i+1}/\left( s_{i}\sqrt{n-i-1}\right) \right) .A.B+O_{\mathfrak{P}%
_{n}}\left( \frac{1}{\left( n-i-1\right) ^{3/2}}\right)
\end{equation*}
where $A$ and $B$ are given in (\ref{Adem}) and (\ref{Bdem}).

\section{Final step of the proof of Theorem 1}

We make use of the following version of the law of large numbers for
triangular arrays (see \cite{Taylor1985} Theorem 3.1.3).

\begin{theorem}
\label{TheoremTaylor}Let $X_{i,n}$ ,$1\leq i\leq k$ denote an array of
row-wise real exchangeable r.v's and $\lim_{n\rightarrow\infty}k=\infty.$
Let $\rho_{n}:=EX_{1,n}X_{2,n}.$ Assume that for some finite $\Gamma$ , $%
EX_{1,n}^{2}\leq\Gamma.$ If for some doubly indexed sequence $\left(
a_{i,n}\right) $ such that $\lim_{n\rightarrow\infty}%
\sum_{i=1}^{k}a_{i,n}^{2}=0$ it holds%
\begin{equation*}
\lim_{n\rightarrow\infty}\rho_{n}\left( \sum_{i=1}^{k}a_{i,n}^{2}\right)
^{2}=0
\end{equation*}
then%
\begin{equation*}
\lim_{n\rightarrow\infty}\sum_{i=1}^{k}a_{i,n}X_{i,n}=0
\end{equation*}
in probability.
\end{theorem}

\ Denote

\begin{equation*}
\kappa_{1}^{i}:=\frac{\mu_{3}^{i}}{2s_{i}^{4}}\text{ \ \ , \ \ }%
\kappa_{2}^{i}:=\frac{\mu_{3}^{i}-s_{i}^{4}}{8s_{i}^{4}}+\frac{%
15(\mu_{3}^{i})^{2}}{72s_{i}^{6}},
\end{equation*}
\begin{equation*}
\mu_{1}^{\ast }:=\kappa_{1}^{i}+\frac{a}{s_{i}^{2}}\text{ \ \ , \ \ }%
\mu_{2}^{\ast}:=\kappa_{1}^{i}-\frac{a}{2s_{i}^{2}}.
\end{equation*}

By (\ref{condTilt}), (\ref{num approx fixed i}) and (\ref{PI 0}) 
\begin{equation*}
p(\mathbf{X}_{i+1}=Y_{i+1}|S_{i+1}^{n}=na_{n}-\Sigma_{1}^{i})=\frac{\sqrt {%
n-i}}{\sqrt{n-i-1}}\pi^{m_{i}}\left( \mathbf{X}_{i+1}=Y_{i+1}\right) \frac{%
\mathfrak{n}\left( \frac{-Y_{i+1}}{s_{i}\sqrt{n-i-1}}\right) }{\mathfrak{n}%
(0)}A(i)
\end{equation*}
with 
\begin{equation*}
A(i):=\frac{1+\frac{\mu_{1}^{\ast}Y_{i+1}}{n-i-1}-\frac{\mu_{2}^{\ast}a}{%
n-i-1}-\frac{\kappa_{2}^{i}}{n-i-1}+\frac{o_{\mathfrak{P}_{n}}\left(
\epsilon_{n}\log n\right) }{n-i-1}}{1-\frac{\kappa_{2}^{i}}{n-i}+O_{%
\mathfrak{P}_{n}}\left( \frac{1}{(n-i)^{3/2}}\right) }.
\end{equation*}
We perform a second order expansion in both the numerator and the
denominator of the above expression, which yields

\begin{equation}
A(i)=\exp\left( \frac{\mu_{1}^{\ast}Y_{i+1}}{n-i-1}-\frac{a}{%
2s_{i}^{2}(n-i-1)}\right) \exp\left( -\frac{a\kappa_{1}^{i}}{n-i-1}\right)
\exp\left( \frac{o_{\mathfrak{P}_{n}}\left( \epsilon_{n}\log n\right) }{n-i-1%
}\right) A^{\prime}(i).  \label{A(i)}
\end{equation}

The term $\exp \left( \frac{\mu _{1}^{\ast }Y_{i+1}}{n-i-1}+\frac{a}{%
2s_{i}^{2}(n-i-1)}\right) $ in (\ref{A(i)}) is captured in $g_{i}(\left.
Y_{i+1}\right\vert Y_{1}^{i}).$

\bigskip The term $A^{\prime}(i)$ in (\ref{A(i)}) writes%
\begin{equation*}
A^{\prime}(i):=Q_{1}^{i}.Q_{2}^{i}
\end{equation*}
with 
\begin{equation*}
Q_{1}^{i}:=\exp\left( -\left( 
\begin{array}{c}
\frac{\kappa_{2}^{i}}{(n-i-1)(n-i)}+\frac{(\kappa_{2}^{i})^{2}}{2(n-i)^{2}}+%
\frac{1}{2}\left( \frac{\mu_{1}^{\ast}Y_{i+1}}{n-i-1}-\frac{a\mu_{2}^{\ast }%
}{n-i-1}-\frac{\kappa_{2}^{i}}{n-i-1}\right) ^{2}%
\end{array}
\right) \right)
\end{equation*}
and 
\begin{equation*}
Q_{2}^{i}:=\frac{\exp B_{1}}{\exp B_{2}}
\end{equation*}
where

\begin{align*}
B_{1} & :=\frac{o_{\mathfrak{P}_{n}}(\epsilon_{n}^{2}(\log n)^{2})}{%
(n-i-1)^{2}}+\frac{\mu_{1}^{\ast}Y_{i+1}}{(n-i-1)^{2}}o_{\mathfrak{P}%
_{n}}\left( \epsilon_{n} \log n\right) \\
& +\frac{\mu_{2}^{\ast}a}{(n-i-1)^{2}}o_{\mathfrak{P}_{n}}(\epsilon_{n}\log
n)+\frac{o_{\mathfrak{P}_{n}}(\epsilon_{n}^{2}\left(\log n)\right)^{2}}{(n-i-1)^{2}}+o(u_{1}^{2})
\end{align*}%
\begin{equation*}
B_{2}:=\frac{\kappa_{2}^{i}}{n-i}O_{\mathfrak{P}_{n}}\left( \frac {1}{%
(n-i)^{3/2}}\right) +O_{\mathfrak{P}_{n}}\left( \frac{1}{(n-i)^{3}}\right) +
\end{equation*}%
\begin{equation*}
O_{\mathfrak{P}_{n}}\left( \frac{1}{(n-i)^{3/2}}\right) +o\left( \left( 
\frac{\kappa_{2}^{i}}{n-i}+O_{\mathfrak{P}_{n}}\left( \frac{1}{(n-i)^{3/2}}%
\right) \right) ^{2}\right) .
\end{equation*}
with 
\begin{equation*}
u_{1}=\frac{\mu_{1}^{\ast}Y_{i+1}}{n-i-1}-\frac{\mu_{2}^{\ast}a}{n-i-1}-%
\frac{\kappa_{2}^{i}}{n-i-1}+\frac{o_{\mathfrak{P}_{n}}(\epsilon_{n}\log n)}{%
n-i-1}.
\end{equation*}

We first prove that%
\begin{equation}
\prod \limits_{i=0}^{k-1}A^{\prime}(i)=1+o_{\mathfrak{P}_{n}}(\epsilon_{n}%
\left( \log n\right) ^{2})  \label{cv produit des A'(i)}
\end{equation}
as $n$ tends to infinity.

Since 
\begin{equation*}
p(\mathbf{X}_{1}^{k}=Y_{1}^{k}|S_{i+1}^{n}=na_{n})=\prod
\limits_{i=0}^{k-1}g_{i}\left( \left. Y_{i+1}\right\vert Y_{1}^{i}\right)
\prod\limits_{i=0}^{k-1}A^{\prime }(i)\prod\limits_{i=0}^{k-1}L_{i}
\end{equation*}%
where 
\begin{equation*}
L_{i}:=\frac{C_{i}^{-1}}{\Phi \left( t_{i}\right) }\frac{\sqrt{n-i}}{\sqrt{%
n-i-1}}\exp \left( -\frac{a\kappa _{1}^{i}}{n-i-1}\right)
\end{equation*}%
the \ completion of the proof will follow from 
\begin{equation}
\prod\limits_{i=0}^{k-1}L_{i}=1+o_{\mathfrak{P}_{n}}(\epsilon _{n}\left(
\log n\right) ^{2}).  \label{cv produit des Li}
\end{equation}

The proof of (\ref{cv produit des A'(i)}) is achieved in two steps.

\begin{claim}
$\prod_{i=0}^{k-1}Q_{1}^{i}=1+o_{\mathfrak{P}_{n}}(\epsilon_{n}\left( \log n\right) ^{2}).$
\end{claim}

By Lemma \ref{LemmaMaxm_in} the random terms $\mu_{j}^{i}$ deriving from $%
\pi^{m_{i}}$ satisfy 
\begin{equation*}
max_{1\leq i\leq k}\left\vert \mu_{j}^{i}-\mu_{j}\right\vert =o_{\mathfrak{P}%
_{n}}(1)
\end{equation*}
as $n$ tends to $\infty,$ where $\mu_{j}$ is the $j$-th centered moment of $%
p.$ Therefore we may substitute $\mu_{j}^{i}$ by $\mu_{j}$ in order to check
the convergence of all subsequent series.

Developing $Q1,$ define, for any positive $\beta_{1}$, $\beta_{2}$, $\beta
_{3}$ and $\beta_{4}$ 
\begin{equation*}
A_{n}^{1}:=\left\{ \frac{1}{\epsilon_{n}\left( \log n\right) ^{2}}%
\sum_{i=0}^{k-1}\left\vert \frac{\kappa_{2}^{i}}{(n-i-1)(n-i)}\right\vert
<\beta_{1}\right\} ,
\end{equation*}%
\begin{equation*}
A_{n}^{2}:=\left\{ \frac{1}{\epsilon_{n}\left( \log n\right) ^{2}}%
\sum_{i=0}^{k-1}\left\vert \frac{(\kappa_{2}^{i})^{2}}{(n-i-1)^{2}}%
\right\vert <\beta_{2}\right\} ,
\end{equation*}%
\begin{equation*}
A_{n}^{3}:=\left\{ \frac{1}{\epsilon_{n}\left( \log n\right) ^{2}}%
\sum_{i=0}^{k-1}\left\vert \frac{(\mu_{2}^{\ast}a)^{2}}{(n-i-1)^{2}}%
\right\vert <\beta_{3}\right\}
\end{equation*}
and 
\begin{equation*}
A_{n}^{4}:=\left\{ \frac{1}{\epsilon_{n}\left( \log n\right) ^{2}}%
\sum_{i=0}^{k-1}\left\vert \frac{\mu_{2}^{\ast}\kappa_{2}^{i}a}{(n-i-1)^{2}}%
\right\vert <\beta_{4}\right\} .
\end{equation*}

It clearly holds that 
\begin{equation*}
\lim_{n\rightarrow\infty}\mathfrak{P}_{n}\left( A_{n}^{j}\right) =1;\text{ }%
j=1,...,4.
\end{equation*}

Let for any positive $\beta_{5}$ 
\begin{equation*}
A_{n}^{5}:=\left\{ \frac{1}{\epsilon_{n}\left( \log n\right) ^{2}}%
\sum_{i=0}^{k-1}\left\vert \frac{\kappa_{1}^{i}\kappa_{2}^{i}Y_{i+1}}{%
(n-i-1)^{2}}\right\vert <\beta_{5}\right\} .
\end{equation*}
If $\lim_{n\rightarrow\infty}\mathfrak{P}_{n}\left( A_{n}^{5}\right) =1$,
then $\lim_{n\rightarrow\infty}\mathfrak{P}_{n}\left( A_{n}^{j}\right) ,$ $%
j=6,7$ where

\begin{equation*}
A_{n}^{6}:=\left\{ \frac{1}{\epsilon_{n}\left( \log n\right) ^{2}}%
\sum_{i=0}^{k-1}\left\vert \frac{\mu_{1}^{\ast}\kappa_{2}^{i}Y_{i+1}}{%
(n-i-1)^{2}}\right\vert <\beta_{6}\right\}
\end{equation*}

\begin{equation*}
A_{n}^{7}:=\left\{ \frac{1}{\epsilon_{n}\left( \log n\right) ^{2}}%
\sum_{i=0}^{k-1}\left\vert \frac{\mu_{1}^{\ast}\mu_{2}^{\ast}aY_{i+1}}{%
(n-i-1)^{2}}\right\vert <\beta_{7}\right\} .
\end{equation*}

Apply Theorem \ref{TheoremTaylor} with $X_{i,n}=Y_{i+1}$ and $a_{i,n}=\frac {%
1}{\epsilon_{n}\left( \log n\right) ^{2}(n-i-1)^{2}}.$ By Lemma \ref%
{LemmaMomentsunderE_n}

\begin{equation*}
E_{\mathfrak{P}_{n}}Y_{1}^{2}=s^{2}(0)+a+O\left( \frac{1}{n}\right) .
\end{equation*}
Hence $E_{\mathfrak{P}_{n}}[Y_{1}^{2}]\leq{\Gamma}$ for some finite $\Gamma.$
Further $\rho_{n}=a^{2}+O\left( \frac{1}{n}\right) .$ Both conditions in Theorem 
\ref{TheoremTaylor} are fullfilled.\ Indeed

\begin{equation*}
\lim_{n\rightarrow\infty}\sum_{i=1}^{k}a_{n,i}^{2}=\lim_{n\rightarrow\infty }%
\frac{1}{\epsilon_{n}^{2}\left( \log n\right) ^{4}(n-k)^{3}}=0
\end{equation*}
which holds under (E1), as holds 
\begin{equation*}
\lim_{n\rightarrow\infty}\rho_{n}\left( \sum_{i=1}^{k}a_{n,i}\right)
^{2}=\lim_{n\rightarrow\infty}\frac{a^{2}}{\epsilon_{n}^{2}\left( \log n\right)
^{4}(n-k)^{2}}=0.
\end{equation*}

Therefore, for i between $5$ and $7$, we have 
\begin{equation*}
\lim_{n\rightarrow\infty}\mathfrak{P}_{n}\left( A_{n}^{i}\right) =1.
\end{equation*}

Define for any positive $\beta_{8}$%
\begin{equation*}
A_{n}^{8}:=\left\{ \frac{1}{\epsilon_{n}\left( \log n\right) ^{2}}%
\sum_{i=0}^{k-1}\frac{\left( \mu_{1}^{\ast}\right) ^{2}Y_{i+1}^{2}}{%
(n-i-1)^{2}}<\beta_{8}\right\} .
\end{equation*}

Apply Theorem \ref{TheoremTaylor} with $X_{i,n}=Y_{i+1}^{2}$ and $a_{i,n}=%
\frac{1}{\epsilon_{n}\left( \log n\right) ^{2}(n-i-1)^{2}}.$

It holds 
\begin{equation*}
\lim_{n\rightarrow\infty}\sum_{i=1}^{k}a_{n,i}^{2}=0
\end{equation*}
when (E1) holds.

By Lemma \ref{LemmaMomentsunderE_n}, 
\begin{equation*}
E_{\mathfrak{P}_{n}}Y_{1}^{4}=E_{\pi^{a}}Y_{1}^{4}+O\left( \frac{1}{n}\right)
\end{equation*}
which entails that for some positive constant $\Gamma$ $\ $such that $%
EY_{1}^{4}\leq{\Gamma}<\infty.$ Also%
\begin{equation*}
E_{\mathfrak{P}_{n}}\left( Y_{1}^{2}Y_{2}^{2}\right) =\left(
s^{2}(0)+a\right) \left( s^{2}(0)+a\right) +O\left( \frac{1}{n}\right)
\end{equation*}
and 
\begin{equation*}
\lim_{n\rightarrow\infty}\rho_{n}\left( \frac{1}{\epsilon_{n}\left( \log
n\right) ^{2}}\sum_{i=0}^{k-1}\frac{1}{(n-i-1)^{2}}\right) ^{2}=0
\end{equation*}
which holds under (E1). Hence%
\begin{equation*}
\lim_{n\rightarrow\infty}\mathfrak{P}_{n}\left( A_{n}^{8}\right) =1.
\end{equation*}
It follows that, noting $A_{n}$ the intersection of the events $A_{n}^{i}$ , 
$j=1,...,8$

\begin{equation*}
\lim_{n\rightarrow\infty}\mathfrak{P}_{n}\left( A_{n}\right) =1.
\end{equation*}
To sum up, we have proved that, under (E1), 
\begin{equation*}
Q1=1+o_{\mathfrak{P}_{n}}\left( \epsilon_{n}\left( \log n\right) ^{2}\right)
.
\end{equation*}

\begin{claim}
$\prod_{i=0}^{k-1}Q_{2}^{i}=1+o_{\mathfrak{P}_{n}}\left( \epsilon_{n}\left( \log n\right)^{2}\right) .$
\end{claim}

This amounts to prove that the sum of the terms in $B1$ (resp in $B2$) is of
order $o_{\mathfrak{P}_{n}}\left( \epsilon_{n}\left( \log n\right)
^{2}\right) .$

The four terms in the the sum of the terms in $B1$ are respectively of order 
$o_{\mathfrak{P}_{n}}\left( \epsilon_{n}^{2}(\log n)^{4}\right) /(n-k)$, $o_{%
\mathfrak{P}_{n}}\left( \epsilon_{n}(\log n)^{3}\right) /(n-k),$ $o_{%
\mathfrak{P}_{n}}\left( a\epsilon_{n}(\log n)^{2}\right) /(n-k)$ and $o_{%
\mathfrak{P}_{n}}\left( \epsilon_{n}(\log n)^{2}\right) /(n-k)$ using Lemma %
\ref{LemmaMaxm_in}. The sum of the terms $o\left( u_{1}^{2}\right) $ is of
order less than those ones.\ Assuming (E1) all those terms are $o_{\mathfrak{%
P}_{n}}\left( \epsilon_{n}\left( \log n\right) ^{2}\right) .$

For the sum of terms $B2$, by uniformity of the Edgeworth expansion with
respect to $Y_{1}^{k}$ it holds $\sum_{i=1}^{k}B2=O_{\mathfrak{P}_{n}}\left(
\left( n-k\right) ^{-1/2}\right) $ which is $o_{\mathfrak{P}_{n}}\left(
\epsilon_{n}\left( \log n\right) ^{2}\right) $ by (E1).

\bigskip

We now turn to the proof of (\ref{cv produit des Li})

Define 
\begin{equation*}
u:=-x\frac{\mu_{3}^{i}}{2s_{i}^{4}\left( n-i-1\right) }+\frac{(x-a)^{2}}{%
2s_{i}^{2}\left( n-i-1\right) }.
\end{equation*}
Use the classical bounds 
\begin{equation*}
1-u+\frac{u^{2}}{2}-\frac{u^{3}}{6}\leq e^{-u}\leq1-u+\frac{u^{2}}{2}
\end{equation*}
to obtain on both sides of the above inequalities the second order
approximation of $C_{i}^{-1}$ through integration with respect to $p.$ The
upper bound yields

\begin{align*}
C_{i}^{-1} & \leq\Phi(t_{i})+\frac{\kappa_{1}^{i}}{n-i-1}\Phi^{%
\prime}(t_{i})+\frac{1}{s_{i}^{2}(n-i-1)}\left(
\Phi"(t_{i})-2a\Phi^{\prime}\left( t_{i}\right) +a^{2}\right) \\
& +O_{\mathfrak{P}_{n}}\left( \frac{1}{(n-i-1)^{2}}\right)
\end{align*}
from which 
\begin{equation*}
L_{i}\leq\frac{\sqrt{n-i}}{\sqrt{n-i-1}}\exp\left( -\frac{a\kappa_{1}^{i}}{%
n-i-1}\right) \left( 
\begin{array}{c}
1+\frac{\kappa_{1}^{i}}{n-i-1}m_{i} \\ 
-\frac{s_{i}^{2}+m_{i}^{2}-2am_{i}+a^{2}}{2s_{i}^{2}\left( n-i-1\right) }+O_{%
\mathfrak{P}_{n}}\left( \frac{1}{\left( n-i-1\right) ^{2}}\right)%
\end{array}
\right)
\end{equation*}
where the approximation term is uniform on the $Y_{1}^{k}.$

Subsituting $\frac{\sqrt{n-i}}{\sqrt{n-i-1}}$ and $\exp\left( -\frac {%
a\kappa_{1}^{i}}{n-i-1}\right) $ by their expansion $1+\frac{1}{2\left(
n-i-1\right) }+O\left( \frac{1}{\left( n-i-1\right) ^{2}}\right) $ and $1-%
\frac{a\kappa_{1}^{i}}{n-i-1}+\frac{(a\kappa_{1}^{i})^{2}}{(n-i-1)^{2}}%
+O\left( \frac{a^{2}}{(n-i-1)^{2}}\right) $ in the upper bound of $L_{i}$ \
above yields

\begin{align*}
L_{i} & \leq\ \left( 1+\frac{1}{2(n-i-1)}-\frac{a\kappa_{1}^{i}}{n-i-1}+%
\frac{(a\kappa_{1}^{i})^{2}}{2(n-i-1)^{2}}+o\left( \frac {1}{(n-i-1)^{2}}%
\right) \right) \\
& \left( 1+\frac{\kappa_{1}^{i}m_{i}}{n-i-1}-\frac{%
s_{i}^{2}+m_{i}^{2}-2am_{i}+a^{2}}{2s_{i}^{2}(n-i-1)}+O_{\mathfrak{P}%
_{n}}\left( \frac {1}{(n-i-1)^{2}}\right) \right) .
\end{align*}

Using Lemma \ref{LemmaMaxm_in}, $m_{i}^{2}-2am_{i}+a^{2}=o_{\mathfrak{P}%
_{n}}(a\epsilon_{n})$ and therefore

\begin{align*}
L_{i} & \leq\ \left( 1+\frac{1}{2(n-i-1)}-\frac{a\kappa_{1}^{i}}{n-i-1}+%
\frac{(a\kappa_{1}^{i})^{2}}{(n-i-1)^{2}}+o\left( \frac {1}{(n-i-1)^{2}}%
\right) \right) \\
& \left( 1+\frac{\kappa_{1}^{i}a}{n-i-1}-\frac{1}{2(n-i-1)}+\frac {o_{%
\mathfrak{P}_{n}}(a\epsilon_{n})}{n-i-1}\right) .
\end{align*}
Write 
\begin{equation*}
\prod_{i=1}^{k}L_{i}\leq{\prod_{i=1}^{k}}\left( {1+M_{i}}\right)
\end{equation*}
with 
\begin{equation*}
M_{i}=\frac{(a\kappa_{1}^{i})^{2}}{(n-i-1)^{2}}+\frac{o_{\mathfrak{P}%
_{n}}(a\epsilon_{n})}{n-i-1}.
\end{equation*}

Under (A) and (E1), $\sum_{i=0}^{k-1}M_{i}$ is $o_{\mathfrak{P}_{n}}\left(
\epsilon_{n}\left( \log n\right) ^{2}\right) .$ This closes the proof of the
Theorem.

\end{document}